\documentclass[12pt]{amsart}
\usepackage{amsmath}
\usepackage{amssymb}
\usepackage{amscd}
\usepackage{graphicx,mathdots}
\usepackage{bm}
\newtheorem{theorem}{Theorem}[section]
\newtheorem{lemma}[theorem]{Lemma}
\newtheorem{proposition}[theorem]{Proposition}
\newtheorem{corollary}[theorem]{Corollary}

\theoremstyle{definition}
\newtheorem{definition}[theorem]{Definition}
\newtheorem{example}[theorem]{Example}
\numberwithin{equation}{section}

\newtheorem{remark}[theorem]{Remark}

\title
[Almost duality for Saito structure and complex reflection groups II]{
Almost duality for Saito structure and complex reflection groups II: 
the case of Coxeter and Shephard groups}
\author{Yukiko Konishi}
\address{ 
Department of Mathematics, 
College of Liberal Arts,
Tsuda University,
Toyko 187-8577, Japan 
}
\email{konishi@tsuda.ac.jp}

\author{Satoshi Minabe}
\address{Department of Mathematics, 
Tokyo Denki University, Tokyo 120-8551,  Japan}
\email{minabe@mail.dendai.ac.jp}

\begin{document}
\maketitle
\begin{abstract}
This article is a sequel to \cite{KMS2018}.
It is known that the orbit spaces of the finite Coxeter groups and the Shephard groups 
admit two types of Saito structures without metric.
One is the underlying structures of the Frobenius structures
constructed by Saito \cite{Saito1993} 
and Dubrovin \cite{Dubrovin2004}.
The other is the natural Saito constructed by Kato--Mano--Sekiguchi \cite{KatoManoSekiguchi2015} 
and by Arsie--Lorenzoni \cite{Arsie-Lorenzoni2016}.
We study the relationship between these two 
Saito structures from the viewpoint of almost duality.  
\end{abstract}
\section{Introduction} 
In 1979, K. Saito constructed 
flat structures on the orbit spaces of the
finite Coxeter groups (i.e. the finite real reflection groups) \cite{Saito1993}. See also \cite{SaitoSekiguchiYano}. 
Nowadays, his flat structure is called the Frobenius structure \cite{Dubrovin1993}.
Generalizing Saito's 
results, Dubrovin 
constructed  Frobenius structures on the orbit spaces of
the Shephard groups \cite{Dubrovin2004}. 
(A Shephard group is 
the symmetric group of  a regular complex polytope.)
We call these 
Frobenius structures
the Coxeter--Shephard (or CS) Frobenius structures.
Dubrovin gave a characterization of the CS Frobenius structures
using his almost duality for Frobenius structures \cite{Dubrovin2004}.

In this article, we call the Saito structure without metric 
the Saito structure for short.
A Frobenius structure has a Saito structure
as an underlying structure \cite{Sabbah}.
In other words, a Frobenius structure is a Saito structure 
together with a compatible metric.

There is a distinguished class of finite complex reflection groups
called the duality groups. It includes the finite Coxeter groups and the Shephard groups.
See \cite[\S B.4]{OrlikTerao} and also Table \ref{fcrg}.
Recently, Kato, Mano and Sekiguchi  showed the existence of 
Saito structures on the orbit spaces of the duality groups \cite{KatoManoSekiguchi2015}.   
Arsie and Lorenzoni also studied the same Saito  structures 
for the duality groups of rank $n=2,3$
\cite{Arsie-Lorenzoni2016}. 
In \cite{KMS2018}, we formulated the almost duality 
for the Saito structure and characterized their Saito structure.
We call it the {\it natural} Saito structure
because it comes from the trivial connection.

So the orbit space of  
a finite Coxeter group or a Shephard group 
has both the CS Frobenius structure and 
the natural Saito structure.
A natural question is that whether the latter 
is the underlying Saito structure of the former.
Arsie and Lorenzoni obtained results on this problem
for rank $n=2,3$ 
\cite{Arsie-Lorenzoni2016}.

In this article, we revisit the construction
of the CS Frobenius structure from the viewpoint of
the almost duality of the Saito structure (Theorem \ref{theorem1}) and
show that 
the multiplication of the natural Saito structure
and that of the CS Frobenius structure 
agree for all the finite Coxeter groups and all the Shephard groups (Corollary \ref{cor}).
To prove this, we do not use the classification of these groups explicitly,
but use their characterization by degrees (see the condition (CS3) in \S \ref{section:CS-groups}).
We also compare
the connections and find that they coincide only for all the finite Coxeter groups and some of the Shephard groups (Theorem \ref{theorem2}).  
We need the classification to prove this theorem.
In the case of  rank $n=2,3$, 
our results are in accord with \cite{Arsie-Lorenzoni2016}.

Moreover we find that 
the natural Saito structure admits a compatible metric 
if and only if it agrees with the underlying Saito structure of the CS Frobenius structure
(Theorem \ref{theorem3}).  
Again, the proof does not use the classification explicitly.

The article is organized as follows.
In \S \ref{review}, we first recall the definitions of 
the Saito structure, the Frobenius structure and the
almost duality.
Then in \S \ref{review-duality-group}, 
we summarize the natural (almost) Saito structure 
for the duality groups.
In \S \ref{section:CS-groups},  
we explain the CS Frobenius structure 
for the finite Coxeter  groups and the Shephard groups
from the viewpoint of the almost duality for Saito structures
(Theorem \ref{theorem1}).
\S \ref{main-results}
contains the main results of this article,
Theorem \ref{theorem2} and Theorem \ref{theorem3}.
The remaining sections are devoted to proofs.
 \S \ref{matrix-representation} is a preliminary: we write
 down conditions in the matrix form with respect to
 flat coordinates of the natural Saito structure.
 In \S \ref{section:proof-multiplication}, we give a proof of Theorem \ref{theorem1} and Theorem \ref{theorem:multiplication}.
In \S \ref{proof-theorem3}, we prove Theorem \ref{theorem3}.
In \S \ref{appendix1}, we describes some technical details of the example $G(m,1,n)$.

\renewcommand{\figurename}{Table}
\begin{figure}
\begin{picture}(400,300)
\put(20,0){\line(1,0){360}}
\put(0,20){\line(0,1){260}}
\put(400,20){\line(0,1){260}}
\put(20,300){\line(1,0){360}}
\put(20,20){\oval(40,40)[bl]}
\put(380,20){\oval(40,40)[br]}
\put(20,280){\oval(40,40)[tl]}
\put(380,280){\oval(40,40)[tr]}
\put(20,280){$G(m,p,n)$ $(1<p<m,n\geq 2),
\quad G_7,G_{11},G_{12},G_{13},G_{15},G_{19},G_{22},G_{31}$}
\put(160,255){Duality groups}
\put(30,10){\line(1,0){340}}
\put(10,30){\line(0,1){210}}
\put(390,30){\line(0,1){210}}
\put(30,260){\line(1,0){130}}
\put(240,260){\line(1,0){130}}
\put(30,30){\oval(40,40)[bl]}
\put(370,30){\oval(40,40)[br]}
\put(370,240){\oval(40,40)[tr]}
\put(30,240){\oval(40,40)[tl]}
\put(80,235){$G(m,m,n)$ $(m,n\geq 3)$, 
$\quad G_{24},G_{27},G_{29},G_{33},G_{34}$}
\put(150,215){finite Coxeter groups}
\put(85,220){\line(1,0){60}}
\put(260,220){\line(1,0){50}}
\put(90,85){\line(1,0){220}}
\put(70,105){\line(0,1){95}}
\put(330,105){\line(0,1){95}}
\put(90,105){\oval(40,40)[bl]}
\put(90,200){\oval(40,40)[tl]}
\put(310,105){\oval(40,40)[br]}
\put(310,200){\oval(40,40)[tr]}
\put(120,195){$G(2,2,n)=D_n$ $(n\geq 4)$}
\put(100,175){
$ G_{35}=E_6,~ G_{36}=E_7, ~G_{37}=E_8$}
\put(100,140){$A_{n-1}$ $(n\geq 2)$, $\quad 
G(2,1,n)=B_n$ ($n\geq 2$) }
\put(120,120){
$ G(m,m,2)=I_2(m)$ $(m\geq 5)$}
\put(110,100){$G_{23}=H_3,~G_{28}=F_4,~G_{30}=H_4$}
\put(160,25){Shephard groups}
\put(40,30){\line(1,0){120}}
\put(250,30){\line(1,0){110}}
\put(40,165){\line(1,0){320}}
\put(20,50){\line(0,1){95}}
\put(380,50){\line(0,1){95}}
\put(40,50){\oval(40,40)[bl]}
\put(360,50){\oval(40,40)[br]}
\put(360,145){\oval(40,40)[tr]}
\put(40,145){\oval(40,40)[tl]}
\put(85,65){$G_3=\mathbb{Z}_m$ $(m\geq 2)$, $G(m,1,n)$ $(m\geq 3,n\geq 2)$}
\put(40,45){$G_4,G_5,G_6$, $G_8,G_9,G_{10},G_{14},
G_{16},G_{17},G_{18},G_{20},G_{21}$, $G_{25},G_{26},G_{32}$}
\end{picture}
\caption{Irreducible finite complex reflection groups \cite[\S B.4]{OrlikTerao}. Notations follow \cite{ShephardTodd}.}
\label{fcrg}
\end{figure}

\subsection*{Acknowledgements}
The work of Y.K is supported in part 
JSPS KAKENHI Kiban-S 16H06337.
The work of S.M. is supported in part by 
JSPS KAKENHI Kiban-C 17K05228.

\section{(Almost) Saito structure and (almost) Frobenius structure}\label{review}

The definition of Saito structure (without metric) can be found in \cite{Sabbah}. See also \cite{KMS2018} for almost Saito structures.
\begin{definition}
A Saito structure (SS for short) on a manifold $M$ 
consists of 
\begin{itemize}
\item a torsion-free flat connection $\nabla$ on $TM$,
\item an associative commutative  multiplication $\ast$ 
on $TM$ with a unit $e\in \Gamma(M,\mathcal{T}_M)$,
\item a vector field $E\in \Gamma(M,\mathcal{T}_M)$ 
called {\it the Euler vector field},  
\end{itemize}
satisfying the following conditions. Let $X,Y,Z\in \mathcal{T}_M$:
\begin{eqnarray}
\nonumber (\bf{SS1})
&&\nabla_X (Y\ast Z)-Y\ast \nabla_X \,Z-\nabla_Y(X\ast Z)
+X\ast \nabla_Y \,Z=[X,Y]\ast Z~.
\\
\nonumber (\bf{SS2})
&&[E,X\ast Y]-[E,X]\ast Y-X\ast[E,Y]=X\ast Y~.
\\
\nonumber (\bf{SS3})
&&\nabla e=0~.
\\
\nonumber (\bf{SS4})
&&
\nabla_X\nabla_Y E-\nabla_{\nabla_X Y}E=0~.
\end{eqnarray}
\end{definition} 

\begin{definition}
An almost Saito structure (ASS for short) on a manifold $N$ 
with parameter $r\in \mathbb{C}$
consists of 
\begin{itemize}
\item a torsion-free flat connection $\boldsymbol{\nabla}$ on $TN$,
\item an associative commutative multiplication 
$\star$ 
on $TN$ with a unit $E\in \Gamma(N,\mathcal{T}_N)$, 
\item a nonzero vector field  $e\in \Gamma(N,\mathcal{T}_N)$
\end{itemize}
satisfying the following conditions. Let $X,Y,Z\in \mathcal{T}_N$:
\begin{eqnarray}
\nonumber ({\bf ASS1})
&&
\boldsymbol{\nabla}_X (Y\star Z)-Y\star \boldsymbol{\nabla}_X \,Z
-\boldsymbol{\nabla}_Y(X\star Z)+X\star \boldsymbol{\nabla}_Y\,Z=[X,Y]\star Z~.
\\
\nonumber ({\bf ASS2})
&& 
[e,X\star Y]-[e,X]\star Y-X\star[e,Y]+e\star X\star Y=0~.
\\
\nonumber ({\bf ASS3})
&&
\boldsymbol{\nabla}_X E=r X ~.
\\
\nonumber ({\bf ASS4})
&&
\boldsymbol{\nabla}_X\boldsymbol{\nabla}_Y \,e
-\boldsymbol{\nabla}_{\boldsymbol{\nabla}_X Y}\,e+
\boldsymbol{\nabla}_{X\star Y} \,e=0~.
\end{eqnarray}
\end{definition}

There is a following relationship between the Saito structure
and the almost Saito structure \cite[Proposition 3.7]{KMS2018}.
Let  $(\boldsymbol{\nabla},\star, e)$ be an ASS on $N$ with the unit $E$
and parameter $r$. 
For a point $p\in N$,
let $\mathcal{P}_p=e\star:T_pN\to T_p N$ and 
$$N_0:=\{p\in N\mid \text{$\mathcal{P}_p$ is invertible}\}~.$$
Then 
if we define a multiplication $\ast$ and 
a connection ${\nabla}$ by
\begin{eqnarray}\label{ass2ss-1}
e\star(X\ast Y)&=&X\star Y~,
\\
\label{ass2ss-2}
\nabla_X \,Y&=&\boldsymbol{\nabla}_X\,Y
-\boldsymbol{\nabla}_{X\ast Y}\,e
~,
\end{eqnarray}
then $e$ is the unit of $\ast$ and 
$({\nabla}, \ast,E)$ is a  SS on $N_0$.
Moreover, it holds that
\begin{eqnarray}\label{ss2ass0}
E\ast(X\star Y)&= &X\ast Y~,\\
\label{ss2ass}
\boldsymbol{\nabla}_X\, Y&=&{\nabla}_X\,Y+r
X\star Y-\nabla_{X\star Y}\,E
~.
\end{eqnarray}
We say that the SS $(\nabla,\ast,E)$ is dual to the ASS $(\boldsymbol{\nabla},\star,e)$.

\begin{remark}
Given a SS $(\nabla,\ast,E)$ with the unit $e$, 
one can make a dual ASS $(\boldsymbol{\nabla},\star,e)$ 
with the unit $E$ by
\eqref{ss2ass0} and \eqref{ss2ass}.
Notice that
there exists a one-parameter family of dual ASS's depending on 
the choice of the parameter $r\in \mathbb{C}$.
\end{remark}

A Frobenius structure \cite{Dubrovin1993} on a manifold $M$
of charge $D\in \mathbb{C}$ is 
a Saito structure $(\nabla,\ast ,E)$ on $M$
together with a nondegenerate symmetric bilinear form (``metric'') $\eta$ on $TM$
satisfying the following conditions. 
Let $X,Y,Z\in \mathcal{T}_M$:
\begin{eqnarray}
\label{f1}
&&X(\eta(Y,Z))=\eta(\nabla_X\,Y,Z)+\eta(Y,\nabla_X\,Z)~.
\\\label{f2}
&&
\eta(X\ast Y,Z)=\eta(X,Y\ast Z) ~.
\\
\label{f3}
&&E\eta(X,Y)-\eta([E,X],Y)-\eta(X,[E,Y])=(2-D)\eta(X,Y)~.
\end{eqnarray}
Note that \eqref{f1} means that $\nabla$ is the Levi--Civita connection of $\eta$
(i.e. the unique torsion free connection on $TM$
compatible with $\eta$).

An almost Frobenius structure \cite[\S 3]{Dubrovin2004}  of charge
$D\in \mathbb{C}$ on a manifold $N$ is 
an almost Saito structure $(\boldsymbol{\nabla},\star ,e)$ 
with parameter 
\begin{equation}\nonumber 
r=\frac{1-D}{2}
\end{equation}
together with a metric  
$g$ on $TN$
satisfying the following conditions. Let $X,Y,Z\in \mathcal{T}_N$:
\begin{eqnarray}
\label{af1}
&&
X(g(Y,Z))=g(\boldsymbol{\nabla}_X\,Y,Z)+g(Y,\boldsymbol{\nabla}_X\,Z)~.
\\\label{af2}
&&g(X\star Y,Z)=g(X,Y\star Z)~.
\\
\label{af3}
&&e \,g(X,Y)-g([e,X],Y)-g(X,[e,Y])+g(e\star X,Y)=0~.
\end{eqnarray}

There is a following relationship between the Frobenius structure
and the almost Frobenius structure \cite{Dubrovin2004}.
Let $(g,\star,e)$ be an almost  Frobenius structure on $N$ of charge $D$ with the unit $E$.
Let us define a multiplication $\ast $ by \eqref{ass2ss-1}.
If we define a metric $\eta$ by
\begin{equation}\label{afs2fs}
\eta(X,Y)=g(X,E\ast Y)~,
\end{equation}
then $(\nabla,\ast,E)$ is a Frobenius structure on $N_0$
of the same charge $D$.  Moreover,
the Levi--Civita connections $\boldsymbol{\nabla}$ and $\nabla$
of $g$ and $\eta$ are related by \eqref{ass2ss-2}.
We say that $(\eta,\ast, E)$ is dual to the almost Frobenius structure $(g,\star,e)$.

\section{The natural Saito structure for duality groups}
\label{review-duality-group}
\subsection{Finite complex reflection groups}
For finite complex reflection groups,  see \cite{LehrerTaylor}
and \cite{OrlikTerao}. 

Let $V=\mathbb{C}^n$ and denote by $u^1,\ldots,u^n$ the standard coordinates of $V$.
Let $G$ be a finite complex reflection group acting on $V$.
It is well known that 
the ring of $G$-invariant polynomials $\mathbb{C}[V]^G=\mathbb{C}[u]^G$ is generated by $n$ $G$-invariant homogeneous polynomials.  
Such a set of generators $x^1,\ldots, x^n$ is called {\it a  set of basic invariants} for 
$G$. 
We assume that $x^1,\ldots, x^n$ are ordered so that {\it the degrees} $d_{\alpha}=\mathrm{deg}\,x^{\alpha}$ ($1\leq \alpha\leq n$) are in descending order, i.e.,
$$
d_1\geq d_2\geq d_3\geq  \ldots \geq d_n~.
$$

The $\mathbb{C}[V]^G$-module of 
$G$-invariant differential $1$-forms on $V$ is denoted 
$\Omega_{\mathbb{C}[V]}^G$. It is a free $\mathbb{C}[V]^G$-module of rank $n$
and $dx^1,\ldots, dx^n$ form its basis (see \cite[Theorem 6.49]{OrlikTerao}).
The $\mathbb{C}[V]^G$-module  of $G$-invariant derivations on $V$ is denoted  $\mathrm{Der}_{\mathbb{C}[V]}^G$. It is also a free
$\mathbb{C}[V]^G$-module of rank $n$
(see \cite[Lemma 6.48]{OrlikTerao}).
A homogeneous basis
$\{X_1,\ldots, X_n\}$
 of $\mathrm{Der}_{\mathbb{C}[V]}^G$
is called {\it a set of basic derivations} for $G$.
The degrees\footnote{
In this article, 
the degree of $\frac{\partial}{\partial u^{\alpha}}$ is counted as 
$-\mathrm{deg}\,u^{\alpha}=-1$.
If $f\in \mathbb{C}[u]$ is a homogeneous polynomial of degree $d$, 
the degree of the vector field 
$f\frac{\partial}{\partial u^{\alpha}}$ is 
$d-1$.}  
$d_1^{\ast},\ldots, d_n^{\ast}$
of $X_1,\ldots, X_n$
are called {\it the codegrees} of $G$. 
When necessary, we order $X_1,\ldots, X_n$ so that 
the codegrees are in ascending order:
$$0=d_1^{\ast}\leq d_2^{\ast}\leq \ldots \leq d_n^{\ast}~.$$

A polynomial $f\in \mathbb{C}[V]$ defines
a homomorphism 
$\mathrm{Hess}(f)$ 
from the $\mathbb{C}[V]$-module of  derivations to the $\mathbb{C}[V]$-module of differential $1$-forms by
$$
\mathrm{Hess}(f)\left(\frac{\partial}{\partial u^i}\right)
=\sum_{j=1}^n \frac{\partial^2 f}{\partial u^i \partial u^j}du^j~.
$$
If $f$ is $G$-invariant, this homomorphism induces 
a map from $\Omega_{\mathbb{C}[V]}^G$ 
to  $\mathrm{Der}_{\mathbb{C}[V]}^G$ (see \cite[Lemma 6.9]{OrlikTerao}).

Let $M=\mathrm{Spec}\,\mathbb{C}[V]^G
=\mathrm{Spec}\,\mathbb{C}[x]
\cong \mathbb{C}^n$
be the orbit space of $G$ and let
$\pi:V\to M$ be the orbit map.
The complement of reflection hyperplanes  is denoted 
$V^{\circ}$  and its image $\pi(V)$ is denoted $M^{\circ}$.
The orbit map $\pi:V^{\circ}\to M^{\circ}$ is an unbranched covering map. 
So we can regard the standard coordinates $u^1,\ldots, u^n$ of $V$
as local coordinates of $M^{\circ}$. 
We will use the two (local) coordinate systems
$x=(x^1,\ldots, x^n)$ and $u=(u^1,\ldots, u^n)$ on $M^{\circ}$.

Since $\pi:V^{\circ}\to M^{\circ}$ is locally a homeomorphism,
the trivial connection on $TV$ 
induces 
a connection  $\boldsymbol{\nabla}^V$
on $TM^{\circ}$. In the local $u$-coordinates,  
it is given by
\begin{equation}\label{trivial-connection-u}
\boldsymbol{\nabla}^V_{\frac{\partial}{\partial u^i}}\frac{\partial}{\partial u^j}=0\quad (1\leq i,j\leq n)~.
\end{equation}
By definition, $\boldsymbol{\nabla}^V$ is flat 
and torsion free.

In this article, we only treat 
the finite complex reflection groups $G$
which are irreducible  (i.e. $G$ acting on $V$ 
irreducibly) and which satisfy the strict inequality 
$d_1>d_2$.
The irreducibility implies 
$d_{\alpha}\geq 2$ ($1\leq \alpha\leq n$).
The inequality $d_1>d_2$ implies that 
the vector field 
\begin{equation}\label{unit-e}
e:=\frac{\partial}{\partial x^1}
\left(=\sum_{i=1}^n \frac{\partial u^i}{\partial x^1} \frac{\partial}{\partial u^i} \right)
\end{equation}
on $M$ is independent of the choice of the 
set of basic invariants $x^1,\ldots, x^n$ up to scalar multiplication.\footnote{
Assume that $d_1>d_2$ and that 
$\tilde{x}=(\tilde{x}^1,\ldots,\tilde{x}^n)$ is another set of basic invariants for $G$. Then by degree consideration,
$$
\tilde{x}^1=a x^1+\text{a polynomial in $x^2,\ldots, x^n$}
\quad (a\in \mathbb{C},~a\neq 0)
$$
and $\tilde{x}^2,\ldots,\tilde{x}^n$ are polynomials in $x^2,\ldots,x^n$.  Therefore  by the chain rule, we have
$$
\frac{\partial}{\partial x^1}=\sum_{\alpha=1}^n
\frac{\partial \tilde{x}^\alpha}{\partial x^1}
\frac{\partial}{\partial \tilde{x}^\alpha}
=a \frac{\partial}{\partial \tilde{x}^1 }~.
$$
}
We also set
\begin{equation}\label{euler-vf}
E=\frac{1}{d_1}E_{\mathrm{deg}}~,\quad 
E_{\mathrm{deg}}
=\sum_{i=1}^n u^i \frac{\partial}{\partial u^i}
=\sum_{\alpha=1}^n d_{\alpha}x^{\alpha}\frac{\partial}{\partial x^{\alpha}}~.
\end{equation}
Notice that the derivation $E_{\mathrm{deg}}$
acts  on a homogeneous polynomial $f\in \mathbb{C}[u]$
or $f\in \mathbb{C}[u]^G=\mathbb{C}[x]$
as
\begin{equation}\label{degree-operator}
E_{\mathrm{deg}} (f) =(\mathrm{deg}\,f) f~.
\end{equation}

\subsection{The natural Saito structure for the duality groups}
\label{section:duality-groups}
For an irreducible finite complex reflection group $G$ of rank $n$,
the following conditions are equivalent. See e.g. \cite[Theorem 2.14]{Bessis2006}.
\begin{enumerate} 
\item[(D1)] $d_{\alpha}+d_{\alpha}^{\ast}=d_1$$ (1\leq \alpha\leq n)$.
\item[(D2)]  $G$ is generated by $n$ reflections.
\item[(D3)]  There exists a set of basic invariants such that
the discriminant $\Delta\in \mathbb{C}[x]$ of $G$ is 
a monic polynomial of degree $n$ as a polynomial in $x^1$.
\end{enumerate}
An irreducible finite complex reflection group $G$ 
satisfying these conditions is called a duality group.

Let $G$ be a duality group
and let $x^1,\ldots, x^n$ be a set of basic invariants for $G$.
For a duality group, $d_1>d_2$ holds. 
This follows from the classification. 
Recall that  $\boldsymbol{\nabla}^V$  given in \eqref{trivial-connection-u}
is flat and torsion free.
In \cite{KMS2018}, 
the followings are proved
using the property (D3).

\begin{theorem}\label{thm:KMS}
\begin{enumerate}
\item 
The endomorphism
$$
T_pM^{\circ}\to T_pM^{\circ}~,\quad X\mapsto \boldsymbol{\nabla}^{V}_X\,e
$$
of the tangent space $T_p M^{\circ}$ 
is invertible at every point $p\in M^{\circ}$ \cite[Corollary 7.3]{KMS2018}. Therefore 
the following condition $(${\bf ASS4}$)$ uniquely determines
the multiplication $\star$ on $TM^{\circ}$:
\begin{equation}\label{star-determined}
\boldsymbol{\nabla}^{V}_X\boldsymbol{\nabla}^{V}_Y \,e-
\boldsymbol{\nabla}_{\boldsymbol{\nabla}^{V}_X Y}^{V}\,e+
\boldsymbol{\nabla}_{X\star Y}^{V}\,e=0\qquad (X,Y\in \mathcal{T}_{M^{\circ}})~.
\end{equation}
\item
 The multiplication 
$\star$ is associative and commutative and has the unit
$E$. 
\item 
$(\boldsymbol{\nabla}^{V},\star,e)$ is an ASS
of parameter $\frac{1}{d_1}$ on $M^{\circ}$ \cite[Corollary 7.6]{KMS2018}.
\end{enumerate}
\end{theorem}
  
\begin{definition}
$(\boldsymbol{\nabla}^V, \star,e)$ is called  a natural ASS for the duality group $G$.
\end{definition}

For convenience of the next examples, let us write the statements (1) and (2)
in Theorem \ref{thm:KMS} 
in the $u$-coordinates.
Denote the structure constants of the multiplication 
$\star$ with respect to 
$\frac{\partial}{\partial u^1},\ldots ,\frac{\partial}{\partial u^n}$
by $\tilde{B}_{ij}^k$, i.e.
$$
\frac{\partial}{\partial u^i}\star\frac{\partial}{\partial u^j}
=\sum_{k=1}^n \tilde{B}_{ij}^k \frac{\partial}{\partial u^k}
\quad (1\leq i,j\leq n)~.
$$
Let us set 
$$
e^k=\frac{\partial u^k}{\partial x^1} ~,
\quad {Q^{k}}_j=\frac{\partial e^k}{\partial u^j}~.
$$
The matrix $Q=({Q^{k}}_j)$  is nothing but the representation matrix of 
the map $X\mapsto \boldsymbol{\nabla}^V_X \,e$
with respect to the basis $\frac{\partial}{\partial u^1},\ldots ,\frac{\partial}{\partial u^n}$.
Therefore the statement (1) is equivalent to the condition $\det Q\neq 0$. 
Eq. \eqref{star-determined} is equivalent to 
\begin{equation}\label{B-u}
\frac{\partial^2 e^k}{\partial u^i\partial u^j}
+\sum_{l=1}^n
\frac{\partial e^k}{\partial u^l} \tilde{B}_{ij}^l=
\frac{\partial {Q^k}_j}{\partial u^i}+\sum_{l=1}^n{Q^k}_l\tilde{B}_{ij}^l
=0
\quad (1\leq i,j,k\leq n)~.
\end{equation}
\begin{example}\label{example-zm-1}
For $G=\mathbb{Z}_m$ ($m\geq 2$), 
$n=1$, $x^1=(u^1)^m$, $d_1=m$. Therefore
$$e^1=\frac{1}{m(u^1)^{m-1}}~,\quad 
\tilde{B}_{11}^1=\frac{m}{u_1}~.
$$
\end{example}

\begin{example} \label{example-gm1n-1}
For $G=G(m,1,n)$ ($m\geq 3$, $n\geq 2$), 
$$
x^{\alpha}={\bf e}_{n+1-\alpha}((u^1)^m,\ldots,(u^n)^m)\quad 
(1\leq \alpha\leq n)~,
$$
where $ {\bf e}_{\alpha}$ denotes the $\alpha$-th elementary symmetric polynomial. 
The vector field $e=\frac{\partial}{\partial x^1}$ is given by
$$e=\sum_{k=1}^n e^k \frac{\partial}{\partial u^k}
~,\quad 
e^k=\frac{(-1)^{n+1}
}{m (u^k)^{m-1}}
\prod_{\begin{subarray}{c}1\leq  j\leq n;\\j\neq k\end{subarray}}
((u^k)^m-(u^j)^m)^{-1}~.
$$
See Corollary \ref{e-gm1n}.
Then it is not difficult to check that 
the following $\tilde{B}_{ij}^k$'s satisfy \eqref{B-u}:
\begin{equation}\nonumber
\begin{split}
\tilde{B}_{ii}^i
&=\sum_{l\neq i} \frac{m(u^i)^{m-1}}{((u^i)^m-(u^l)^m)}+\frac{m}{u_i}~,
\\
\tilde{B}_{ii}^k&=-\frac{m(u^i)^{m-2}u^k}{((u^i)^m-(u^k)^m)} \quad (i\neq k)
\\
\tilde{B}_{ij}^i&=\tilde{B}_{ji}^i=-\frac{m(u^j)^{m-1}}{(u^i)^m-(u^j)^m}
\quad (i\neq j)
\\
\tilde{B}_{ij}^k&=0 \quad (i\neq j\neq k\neq i)~.
\end{split}
\end{equation}
These are the structure constants of $\star$ with respect to
the $u$-coordinates.
\end{example}

Now, let $\ast$ and $\nabla$ be the multiplication 
and the connection on $TM^{\circ}$ dual to 
$\star$ and $\boldsymbol{\nabla}^V$ (see
\eqref{ass2ss-1} and \eqref{ass2ss-2}):
\begin{equation}\label{duality-dual}
e\star(X\ast Y)=X\star Y~,\quad
\nabla_X \,Y=\boldsymbol{\nabla}^V_X\,Y
-\boldsymbol{\nabla}^V_{X\ast Y}\,e~.
\end{equation}
Then $(\nabla,\ast,E)$ is a SS on $M^{\circ}$
and called  a natural SS for $G$. 
Notice that for this case, the converse relations
\eqref{ss2ass0} and \eqref{ss2ass} become
\begin{equation}\label{duality-dual2}
E\ast(X\star Y)= X\ast Y~,\quad 
\boldsymbol{\nabla}^V_X\, Y={\nabla}_X\,Y+\frac{1}{d_1}
X\star Y-\nabla_{X\star Y}\,E
~.
\end{equation}
In  \cite[Theorem 7.5 (3)]{KMS2018},  
the following theorem is proved.
\begin{theorem}\label{duality-polynomial}
The natural SS $(\nabla,\ast,E)$ for $G$ is polynomial i.e.,
\begin{enumerate}
\item[(i)] there exists a system of 
$\nabla$-flat coordinates $t=(t^1,t^2,\ldots, t^n)$ which is a set of basic invariants for $G$, and
\item[(ii)]
the structure constants of the multiplication with respect to the basis
$\frac{\partial}{\partial t^1},\ldots,\frac{\partial}{\partial t^n}$ 
are polynomials in $t$.
\end{enumerate}
\end{theorem}
Therefore the natural SS $(\nabla,\ast, E)$ is 
canonically extended to the whole orbit space  $M$.

\section{The Coxeter--Shephard Frobenius structures}
\label{section:CS-groups}
\subsection{The finite Coxeter groups and the Shephard groups}
The duality groups
include
the finite Coxeter groups\footnote
{Among the duality groups, the finite Coxeter groups are characterized by the
property that $d_n=2$.} and the Shephard groups.
For an irreducible finite complex reflection group $G$, 
the following conditions are equivalent (see \cite[Theorem 6.121]{OrlikTerao}):
\begin{enumerate}
\item[(CS1)] $G$ is a finite Coxeter group or a Shephard group.
\item[(CS2)] $\mathrm{Hess}(x^n): \mathrm{Der}_{\mathbb{C}[V]}^G\to \Omega_{\mathbb{C}[V]}^G$
is an isomorphism.
\item[(CS3)] 
$
d_{\alpha}+d_{n+1-\alpha}=d_1+d_n$ ($1\leq \alpha\leq n$).
\end{enumerate}
In this section, $G$ is a finite Coxeter group
or a Shephard group.
Let $x^1,\ldots,x^n$ be a set of basic invariants for $G$.

From the classification, we can see that
the strict inequalities
\begin{equation}\label{degrees}
d_1>d_2>\ldots>d_n~
\end{equation}
hold for a finite Coxeter group or a Shephard group $G$.
The condition (CS3)
and eq.\,\eqref{degrees}
together imply that 
\begin{equation}\label{cs-degrees}
d_{\alpha}+d_{\beta}>d_1+d_n \quad 
(\alpha+\beta<n+1)~,\quad
d_{\alpha}+d_{\beta}<d_1+d_n 
\quad (\alpha+\beta>n+1)~.
\end{equation}

Now let $h$ be a symmetric bilinear form 
on $TM^{\circ}$
corresponding to the map $\mathrm{Hess}(x^n)$, i.e.
\begin{equation}\label{def-h}
h(\cdot,\cdot)=\langle \mathrm{Hess}(x^n)(\cdot),\cdot \rangle~.
\end{equation}
The condition (CS2) implies that  $h$
is 
a metric on $TM^{\circ}$.
Let $\boldsymbol{\nabla}^{cs}$ be 
the Levi--Civita connection for the metric $h$.
In the local $u$-coordinates,  $h$ is given by
\begin{equation}\label{hessian-metric-u}
\tilde{H}_{ij}:=
h\left(\frac{\partial}{\partial u^i},\frac{\partial}{\partial u^j}
\right)
=
\left\langle
\mathrm{Hess}(x^n) \left(\frac{\partial}{\partial u^i}\right),
\frac{\partial}{\partial u^j}
\right\rangle
=
\frac{\partial^2 x^n}{\partial u^i\partial u^j}~,
\end{equation}
where $\langle~~,~~\rangle$ is the canonical pairing.
The Levi--Civita connection $\boldsymbol{\nabla}^{cs}$ is 
expressed as
\begin{equation}\label{christoffel-u}
\begin{split}
\boldsymbol{\nabla}^{cs}_{\frac{\partial}{\partial {u^i}}} \left(\frac{\partial}{\partial u^j}\right)
&=\sum_{k=1}^n \tilde{S}_{ij}^{k} 
\frac{\partial}{\partial u^k}
~,\quad
\tilde{S}_{ij}^k=
\frac{1}{2}\sum_{l=1}^n \tilde{H}^{kl}
\frac{\partial^3 x^n}{\partial u^i\partial u^j\partial u^l}
\end{split}
\end{equation}
Here $\tilde{H}^{ij}$ ($1\leq i,j\leq n$) denotes
the $(i,j)$ entry of the inverse matrix $\tilde{H}^{-1}$
of $\tilde{H}=(\tilde{H}_{ij})$.
From this expression,
we can immediately see that 
$\tilde{S}_{ij}^k=0$ holds
if $d_n =2$.
In other words, $\boldsymbol{\nabla}^{cs}=\boldsymbol{\nabla}^V$ 
if $G$ is a finite Coxeter group.

\begin{example}\label{example-zm-2}
For $G=\mathbb{Z}_m$ ($m\geq 2$), 
$n=1$, $x^1=(u^1)^m$, $d_1=m$. Therefore
$$\tilde{H}_{11}=m(m-1)(u^1)^{m-2}~,\quad 
\tilde{S}_{11}^1=\frac{m-2}{2u^1}~.
$$
\end{example}

\begin{example}\label{example-gm1n-2}
For $G=G(m,1,n)$ $(m\geq 3,n\geq 2)$, 
$x^n=(u^1)^m+\cdots+(u^n)^m$.  Therefore
\begin{equation}\nonumber
\tilde{H}_{ij}=\begin{cases}m(m-1)(u^i)^{m-2} &(i=j)\\0&(i\neq j)\end{cases}~,
\quad
\end{equation}
and 
\begin{equation}\nonumber\label{S-Gm1n}
\tilde{S}_{ij}^k=\begin{cases}\frac{m-2}{2u^i} &(i=j=k)\\
0&(\text{otherwise})\end{cases}~.
\end{equation}
\end{example}

\begin{lemma} \label{cs-flat}
\begin{enumerate}
\item
$\boldsymbol{\nabla}^{cs}$ is flat. 
\item
$\boldsymbol{\nabla}^{cs}$ and $E$ 
$($given in \eqref{degree-operator}$)$
satisfy
$(${\bf ASS3}$)$ with $r=\frac{d_n}{2d_1}$: 
$$
\boldsymbol{\nabla}^{cs}_X E=\frac{d_n}{2d_1} X \quad (X\in 
\mathcal{T}_{M^{\circ}})~.
$$
\end{enumerate}
\end{lemma}
\begin{proof} 
Let us set 
$$
{(\tilde{S}_i)^{k}}_j=\tilde{S}_{ij}^k \quad (1\leq i\leq n)~.
$$
(1) The flatness is equivalent to 
\begin{equation}\label{flatness-u}
\frac{\partial \tilde{S}_j}{\partial u^i} 
-\frac{\partial \tilde{S}_i}{\partial u^j} 
+
\tilde{S}_i\tilde{S}_j-\tilde{S}_j\tilde{S}_i
\stackrel{\eqref{christoffel-u}}{=}
-(\tilde{S}_i\tilde{S}_j-\tilde{S}_j\tilde{S}_i)
=O~.
\end{equation}
If $G$ is a finite Coxeter group (i.e. if $d_n=2$),
then $\tilde{S}_i=O$
 ($1\leq i\leq n$) by \eqref{christoffel-u}.
In the cases  $G=G_3\cong \mathbb{Z}_m$ ($m\geq 1$)
and   $G=G(m,1,n)$ $(m\geq 3,n\geq 2)$, 
it is easy to check that $\tilde{S}_{ij}^k$'s obtained in
Example \ref{example-zm-2} and Example \ref{example-gm1n-2}
satisfy \eqref{flatness-u}.
For the remaining exceptional groups, 
$x^n$ can be found in \cite[Chapter 6]{LehrerTaylor},
\cite[\S B.3]{OrlikTerao}. 
We checked 
that \eqref{flatness-u} holds 
using Mathematica.
\\
(2) Recall that the derivation 
$E_{\mathrm{deg}}$  in \eqref{euler-vf} acts on a
homogeneous  polynomial $f\in \mathbb{C}[x]=\mathbb{C}[u]^G$ by 
$E_{\mathrm{deg}}(f)=(\mathrm{deg}\,f)\,f$ (see
 \eqref{degree-operator}).  So we have
\begin{equation}\nonumber
\begin{split}
\boldsymbol{\nabla}^{cs}_{\frac{\partial}{\partial u^i}}
E&=
\frac{1}{d_1}\left(
\frac{\partial}{\partial u^i}
+\frac{1}{2}\sum_{j,k=1}^n u^j 
\tilde{S}_{ij}^k\frac{\partial}{\partial u^k}
\right)
\\
&\stackrel{\eqref{christoffel-u}}{=}
\frac{1}{d_1}\left(
\frac{\partial}{\partial u^i}
+\frac{1}{2}\sum_{j,k,l=1}^n
u^j \frac{\partial  \tilde{H}_{il}}{\partial u^j} \cdot
\tilde{H}^{kl}
 \frac{\partial}{\partial u^k}
\right)
\\
&=\frac{1}{d_1}\left(
\frac{\partial}{\partial u^i}
+\frac{d_n-2}{2}\sum_{k,l=1}^n
\tilde{H}_{il}
\tilde{H}^{kl}
 \frac{\partial}{\partial u^k}
\right)
\\&=\frac{1}{d_1}\left(
\frac{\partial}{\partial u^i}
+\frac{d_n-2}{2}
 \frac{\partial}{\partial u^i}
\right)
=\frac{d_n}{2d_1} \frac{\partial}{\partial u^i}~.
\end{split}
\end{equation}
\end{proof}

\begin{theorem}\label{theorem1}
\begin{enumerate}
\item The endomorphism
$$
T_pM^{\circ}\to T_pM^{\circ}~,\quad X\mapsto \boldsymbol{\nabla}^{cs}_X\,e
$$
of the tangent space $T_p M^{\circ}$ 
is invertible at every point $p\in M^{\circ}$. Therefore 
the following condition uniquely determines
the multiplication $\diamond$ on $TM^{\circ}$:
\begin{equation}\label{ass4-S}
\boldsymbol{\nabla}^{cs}_X\boldsymbol{\nabla}^{cs}_Y \,e-
\boldsymbol{\nabla}_{\boldsymbol{\nabla}^{cs}_X Y}^{cs}\,e+
\boldsymbol{\nabla}_{X\diamond Y}^{cs}\,e=0\qquad (X,Y\in \mathcal{T}_{M^{\circ}})~.
\end{equation}

\item The multiplication 
$\diamond$ is associative and commutative and has the unit
$E$. 
\item $(\boldsymbol{\nabla}^{cs},\diamond,e)$ is an ASS
with parameter $\frac{d_n}{2d_1}$ on $M^{\circ}$.
\item
$(h,\diamond,e)$ is an almost Frobenius structure of charge 
$1-\frac{d_n}{d_1}$ on $M^{\circ}$.
\end{enumerate}
\end{theorem}
The proof of Theorem \ref{theorem1} will be given in 
\S \ref{section:proof-multiplication}.

Now let $\circledast$ be the multiplication dual to $\diamond$ and 
let $\nabla^{cs}$ be the connection dual to $\boldsymbol{\nabla}^{cs}$ 
(see \eqref{ass2ss-1}, \eqref{ass2ss-2}):
\begin{equation}\label{CS-dual}
e\diamond(X\circledast Y)=X\diamond Y~,
\quad
\nabla^{cs}_X \,Y=\boldsymbol{\nabla}^{cs}_X\,Y
-\boldsymbol{\nabla}^{cs}_{X\circledast Y}\,e~.
\end{equation}
Notice that  the converse relations
\eqref{ss2ass0} and \eqref{ss2ass} become
\begin{equation}\label{CS-dual2}
E\circledast(X\diamond Y)= X\circledast Y~,\quad
\boldsymbol{\nabla}^{cs}_X\, Y={\nabla}^{cs}_X\,Y+\frac{d_n}{2d_1}
X\diamond Y-\nabla_{X\diamond Y}\,E
~.
\end{equation}
Then the multiplication $\circledast$ has $e$ as the unit.
Moreover, by the almost duality,
$(\nabla^{cs},\circledast, E)$ is a SS on $M^{\circ}$.
The following theorem says that 
this SS can be extended to $M$.
\begin{theorem}\label{cs-polynomial}
The SS $(\nabla^{cs},\circledast, E)$ is a polynomial
SS on $M$.
\end{theorem}
The proof  will be given in \S \ref{cs-polynomial-proof}.

Let $\eta$
be the metric dual to  $h$  (see \eqref{afs2fs}):
\begin{equation}\label{CS-dual-2}
\eta(X,Y)=h(X,E\circledast Y)~.
\end{equation}
Then $(\eta,\circledast, E)$ is a Frobenius structure on $M$
of charge  $D=1-\frac{d_n}{d_1}$
which has $(\nabla^{cs},\circledast, E)$ as the underlying Saito structure.
In this article, 
we call $(\eta,\circledast, E)$ the Coxeter--Shephard (CS) Frobenius structure.

\section{Relationship between the two Saito Structures}
\label{main-results}
Let 
$G$ be a finite Coxeter group or a Shephard group.

Recall that 
a finite Coxeter group or a Shephard group is a duality group.
Therefore  we have
two ASS's for $G$. The one is the natural ASS
$(\boldsymbol{\nabla}^V,\star,e)$ with parameter $\frac{1}{d_1}$
explained in \S \ref{section:duality-groups}
and the other is
$(\boldsymbol{\nabla}^{cs},\diamond,e)$ 
with parameter $\frac{d_n}{2d_1}$
explained in \S \ref{section:CS-groups}.
So it is natural to ask whether they are the same or not.
It is clear that
the parameters $\frac{d_n}{2d_1}$ and
$\frac{1}{d_1}$ agree if and only if $d_n=2$ i.e. $G$ is a
finite Coxeter group.
The connection $\boldsymbol{\nabla}^{cs}$
agree with $\boldsymbol{\nabla}^V$ agree if and only if 
$(\tilde{H}_{ij})$ is a constant matrix, 
i.e. $d_n=2$.  
As for the multiplication, 
for any finite Coxeter group or any Shephard group $G$,
we have the following
\begin{theorem}\label{theorem:multiplication}
The multiplication $\star$ and 
 the multiplication $\diamond$ are the same.
\end{theorem}
The proof of Theorem \ref{theorem:multiplication} will be given in \S \ref{proof1}.
\begin{corollary}
The two ASS's $(\boldsymbol{\nabla}^V,\star,e)$ and 
 $(\boldsymbol{\nabla}^{cs},\diamond,e)$ 
for $G$
agree if and only if $d_n=2$, i.e. 
if and only if $G$ is a finite Coxeter group.
\end{corollary}

Similarly, we may ask whether
the two dual Saito structures  $(\nabla,\ast,E)$
and $(\nabla^{cs},\circledast, E)$
are the same or not. 
Theorem \ref{theorem:multiplication} and the first equations of \eqref{duality-dual}, \eqref{CS-dual} imply the following
\begin{corollary}\label{cor}
The multiplication $\circledast$ and the multiplication $\ast$
are the same.
\end{corollary}

As for the connections,
the second equations of \eqref{duality-dual}, \eqref{duality-dual2}, \eqref{CS-dual}, \eqref{CS-dual2} imply
that 
$\nabla^{cs}=\nabla$ holds if and only if 
\begin{equation}\nonumber 
\boldsymbol{\nabla}^{cs}_XY=\boldsymbol{\nabla}^V_XY+\frac{d_n-2}{2d_1}X\star
Y~\quad (X,Y\in \mathcal{T}_M)~.
\end{equation}
In the $u$-coordinates, 
the above relation 
can be expressed as 
\begin{equation}\label{difference}
\tilde{S}_{ij}^k=\frac{d_n-2}{2d_1} \tilde{B}_{ij}^k~.
\end{equation}
Using the classification of the Shepard groups, we obtain the following  
\begin{theorem} \label{theorem2}
\begin{enumerate}
\item
$\nabla=\nabla^{cs}$ holds if and only if $G$ is a finite Coxeter group
or one of the following groups: 
$$
G_3\cong \mathbb{Z}_m, \quad G_4,\quad 
G_5,\quad G_8,\quad G_{16},\quad G_{20},\quad
G_{25},\quad G_{32}~.
$$
\item
The Saito structures $(\nabla,\ast,E)$ and
$(\nabla^{cs},\circledast, E)$
for $G$ agree 
if and only if $G$ is a finite Coxeter group
or one of the above groups.
\end{enumerate}
\end{theorem}
\begin{proof}
If $G$ is a finite Coxeter group, \eqref{difference} is true since 
$\tilde{S}_{ij}^k=0$ and $d_n=2$.
For $G=\mathbb{Z}_m$ ($m\geq 2$), 
we can see \eqref{difference} holds
by comparing Examples \ref{example-zm-1} and \ref{example-zm-2}.
For 
$G=G(m,1,n)$ ($m\geq 3$, $n\geq 2$),
we see 
\eqref{difference} does not hold by 
comparing Examples \ref{example-gm1n-1} and 
\ref{example-gm1n-2}.
For the remaining exceptional Shephard groups, 
we checked 
whether \eqref{difference} holds or not using Mathematica\footnote
{Given a set of basic invariants, it is not difficult to
compute $\tilde{S}_{ij}^k$ and $\tilde{B}_{ij}^k$
using Mathematica. See \eqref{B-u} and \eqref{christoffel-u}. 
Formulas for basic invariants can be found,
e.g., in \cite{LehrerTaylor, OrlikTerao}. For $G_{32}$, see \cite{Maschke}.}
and obtained the result.
\end{proof}

Studying the condition $\nabla=\nabla^{cs}$
in the $x$-coordinates, we are led to the following
\begin{theorem}\label{theorem3}
The natural SS $(\nabla,\ast,E)$ 
admits a compatible Frobenius structure 
if and only if $\nabla=\nabla^{cs}$. 
Moreover, a compatible metric
is a constant multiple of $\eta$
and the charge is $D=1-\frac{d_n}{d_1}~.$
\end{theorem}
The proof will be given in \S \ref{proof-theorem3}.

\section{The matrix representation with respect to
$\nabla$-flat coordinates} 
\label{matrix-representation}

Let $G$ be a finite Coxeter group or a Shephard group of rank $n$
and let $x^1,\ldots,x^n$ be a set of basic invariants for $G$. 
We take $e$ and $h$ as in \eqref{unit-e} and  \eqref{def-h}.
Denote by
$(\nabla, \ast, E)$  the natural Saito structure for $G$ 
with the unit $e$.
For the sake of convenience, we take
a system of $\nabla$-flat coordinates
$t=(t^1,\ldots, t^n)$
satisfying (i) and (ii) in Theorem \ref{duality-polynomial}. 
Here we choose the normalization
$t^1=x^1+\mathbb{C}[x^2,\ldots, x^n]$,
$t^n=x^n$ 
so that 
$e$ and $h$ are unchanged, i.e.
$$ 
e=\frac{\partial}{\partial t^1}~,\quad 
h(\cdot,\cdot)=\langle \mathrm{Hess}(t^n)(\cdot),\cdot \rangle~.
$$ 
Notice that $E_{\mathrm{deg}}$ and $E$
defined in \eqref{euler-vf}
is also written as 
\begin{equation}\label{euler-vf2}
E_{\mathrm{deg}}=\sum_{\alpha=1}^n d_{\alpha}t^{\alpha}\frac{\partial}{\partial t^{\alpha}}~,\quad 
E=\frac{1}{d_1}E_{\mathrm{deg}}~.
\end{equation}
Below we write $\mathbb{C}[t^1,t^2,\ldots t^n]=\mathbb{C}[t]$,
$\mathbb{C}[t']=\mathbb{C}[t^2,\ldots, t^n]$ and 
$$
\partial_{\alpha}=\frac{\partial}{\partial t^{\alpha}}\quad 
(1\leq \alpha\leq n). 
$$
The $\nabla$-flatness is expressed as 
$\nabla (\partial_{\alpha})=0$ ($1\leq \alpha\leq n$). 

\subsection{Matrix representations}
In this subsection, we write
down the conditions for the natural SS $(\nabla,\ast,E)$
and the natural ASS  $(\boldsymbol{\nabla}^V,\star, e)$ 
in the the matrix form 
with respect to the $\nabla$-flat coordinates $t$.
(See also \cite[\S 5]{KMS2018}).

For two $n\times n$ matrices $A$ and $B$,
$[A,B]:=AB-BA$.  The identity matrix 
and the zero matrix are
 denoted $I$ and $O$.
 $M_n(\mathbb{C}[t])$ denote the 
 space of $n\times n$ matrices whose entries are 
 polynomials in $t$.

First, let us consider the conditions for the natural SS
$(\nabla,\ast,E)$ for $G$.
Denote by $C_{\alpha\beta}^{\gamma}$ ($1\leq \alpha,\beta,\gamma\leq n$) 
the structure constants of the multiplication $\ast$:
$$
\partial_{\alpha}\ast\partial_{\beta}=\sum_{\gamma=1}^n
C_{\alpha\beta}^{\gamma}\partial_{\gamma}~,
\quad{ (C_{\alpha})^{\gamma}}_{\beta}:=C_{\alpha\beta}^{\gamma}~.
$$
The matrix $C_{\alpha}$ is the matrix representation of 
$\partial_{\alpha}\ast$ with respect to the basis 
$(\partial_1,\ldots,\partial_{n})$.
Since the multiplication $\ast$ is commutative, associative, and has $e=\partial_1$ as the unit,
we have
\begin{equation}\label{c1}
C_{\alpha\beta}^{\gamma}=C_{\beta\alpha}^{\gamma}~,
\quad [C_{\alpha},C_{\beta}]=O~,\quad 
C_1=I~.
\end{equation}
Denote by $U$
the representation matrix of 
$E\ast$ with respect to  the basis
$\partial_1,\ldots,\partial_n$:
\begin{equation}\label{U-t}
E\ast\partial_{\beta}=\sum_{\gamma=1}^n {U^{\gamma}}_{\beta}\partial_{\gamma}~,\quad 
U=\sum_{\alpha=1}^n \frac{d_{\alpha}}{d_1}t^{\alpha}C_{\alpha}~.
\end{equation}
Notice that the second equation of \eqref{c1} implies 
\begin{equation} \label{uc}
[U,C_{\alpha}]=O \quad (1\leq \alpha\leq n)~.
\end{equation}
Given that $\nabla (\partial_{\alpha})=0$,
the conditions ({\bf SS1}), ({\bf SS2}) are written as follows.
\begin{equation}\label{dU}
\partial_{\alpha}C_{\beta}=\partial_{\beta}C_{\alpha}~,\quad 
\partial_{\alpha}U=WC_{\alpha}-C_{\alpha}W+C_{\alpha}
\quad (1\leq \alpha,\beta\leq n)~,
\end{equation}
where
$$
W=\frac{1}{d_1}\mathrm{diag}(d_1,\ldots,d_n)~.
$$
It is clear that $e=\partial_1$ and $E$ satisfy
the conditions ({\bf SS3}), ({\bf SS4}).

Recall that  the natural SS $(\nabla,\ast, E)$ is polynomial, i.e.,
$C_{\alpha\beta}^{\gamma}\in \mathbb{C}[t]$.  
By the first equation of \eqref{dU} and  by $C_1=I$, we have
$\partial_1 C_{\alpha}=\partial_{\alpha} C_1=O$. 
Therefore 
\begin{equation}\label{indep-t1}
C_{\alpha}\in M_n(\mathbb{C}[t'])~,
\quad 
U- t^1  I \in M_n (\mathbb{C}[t'])~.
\end{equation}
So $\det U\in \mathbb{C}[t]$ is a monic of 
degree $n$ as a polynomial in $t^1$.
In fact, $\det U$ agrees  with the discriminant 
polynomial 
of $G$
and $$M\setminus M^{\circ}=\{t\in M \mid \det U=0 \}~.$$

Applying the derivation $E_{\mathrm{deg}}$ to 
$C_{\alpha\beta}^{\gamma}$, we obtain
\begin{equation}\nonumber
\begin{split}
(\mathrm{deg}\,C_{\alpha\beta}^{\gamma})C_{\alpha\beta}^{\gamma}
&\stackrel{\eqref{degree-operator}}{=}
E_{\mathrm{deg}}C_{\alpha\beta}^{\gamma}
=\sum_{\mu=1}^n d_{\mu}t^{\mu }\partial_{\mu}C_{\alpha\beta}^{\gamma}
\stackrel{\eqref{dU}}{=}
\sum_{\mu=1}^n d_{\mu}t^{\mu }\partial_{\alpha}C_{\mu\beta}^{\gamma}
\\
&=d_1 \partial_{\alpha}U-d_{\alpha}C_{\alpha\beta}^{\gamma}
\stackrel{\eqref{dU}}{=}(d_1+d_{\gamma}-d_{\alpha}-d_{\beta})C_{\alpha\beta}^{\gamma}.
\end{split}
\end{equation}
Therefore 
\begin{equation}\label{degree-C}
\mathrm{deg}\,C_{\alpha\beta}^{\gamma}=d_1+d_{\gamma}-d_{\alpha}-d_{\beta}~,
\quad
\mathrm{deg}\, {U^{\gamma}}_{\beta}=d_1+d_{\gamma}-d_{\beta}~.
\end{equation}

Next let us consider the ASS
$(\boldsymbol{\nabla}^V,\star,e)$ with parameter $r=\frac{1}{d_1}$ which is dual to $(\nabla,\ast E)$.
Denote by 
$B_{\alpha\beta}^{\gamma}$ ($1\leq \alpha,\beta,\gamma\leq 1$) the structure constants of the multiplication $\star$:
$$
\partial_{\alpha}\star\partial_{\beta}=\sum_{\gamma=1}^n
{B_{\alpha\beta}^{\gamma}}\partial_{\gamma}~,
\quad {(B_{\alpha})^{\gamma}}_{\beta}=B_{\alpha\beta}^{\gamma}~.
$$
Substituting $X=\partial_{\alpha}$ and $Y=\partial_{\beta}$
into the first equation of \eqref{duality-dual2}, 
we have
$$
\sum_{\delta=1}^n {U^{\gamma}}_{\delta}B_{\alpha\beta}^{\delta}
=C_{\alpha\beta}^{\gamma}~,$$
or $UB_{\alpha}=C_{\alpha}$. 
Thus, 
\begin{equation}\label{B-C}
B_{\alpha}=U^{-1}C_{\alpha}\stackrel{\eqref{uc}}{=}
C_{\alpha}U^{-1}~.
\end{equation}
(Therefore the entries of $B_{\alpha}$ are 
homogeneous rational functions in $t$ with the 
denominator $\det U$. $B_{\alpha}$ is only defined on $M^{\circ}$.)

Eqs. \eqref{c1} and \eqref{uc} imply
\begin{equation}\label{id-B}
B_{\alpha\beta}^{\gamma}
=B_{\beta\alpha}^{\gamma}~,\quad 
[B_{\alpha},C_{\beta}]=O~,\quad
[B_{\alpha},B_{\beta}]=O~,\quad 
[B_{\alpha},U]=O~.
\end{equation}
Eq. \eqref{dU} together with \eqref{id-B} implies that 
\begin{equation}\label{dB}
\begin{split}
\partial_{\alpha}B_{\beta}
&=
\partial_{\alpha}(C_{\beta}U^{-1})
=(\partial_{\alpha}C_{\beta}) U^{-1}-C_{\beta}U^{-1}
(\partial_{\alpha}U)U^{-1}
\\&=
(\partial_{\alpha}C_{\beta})U^{-1}
-B_{\beta}WB_{\alpha}+U^{-1}C_{\beta}C_{\alpha}W U^{-1}-B_{\beta}B_{\alpha}~.
\end{split}
\end{equation}
Especially, if $\alpha=1$, $C_1=I$ and 
$\partial_1 C_{\beta}=\partial_{\beta} I=O$. So we have
\begin{equation}\label{dB2}
\partial_1 B_{\beta}=-B_{\beta}U^{-1}~.
\end{equation}

Now let $\Omega_{\alpha}$ be 
the connection matrix of $\boldsymbol{\nabla}^V$:
$$
\boldsymbol{\nabla}^V_{\partial_{\alpha}}(\partial_{\beta})=
\sum_{\gamma=1}^n \Omega_{\alpha\beta}^{\gamma}\partial_{\gamma}~,
\quad{(\Omega_{\alpha})^{\gamma}}_{\beta}=\Omega_{\alpha\beta}^{\gamma}~.
$$
Substituting 
$X=\partial_{\alpha}$, $Y=\partial_{\beta}$
into the second equation of 
\eqref{duality-dual2}, 
we obtain
\begin{equation}\label{omega-mr}
\Omega_{\alpha\beta}^{\gamma}=
\frac{1-d_{\gamma}}{d_1} B_{\alpha\beta}^{\gamma}~.
\end{equation}
Moreover, \eqref{trivial-connection-u} implies that 
\begin{equation}\label{trivial-connection}
\Omega_{\alpha\beta}^{\gamma}=-
\sum_{i,j=1}^n 
\frac{\partial u^i}{\partial t^{\alpha}}
\frac{\partial u^j}{\partial t^{\beta}}
\frac{\partial^2 t^{\gamma}}{\partial u^i\partial u^j}
~.
\end{equation}

\subsection{Representation matrix of the metric $h$}
\label{rep-h}
In the $t$-coordinates,   the metric $h$ 
is given by 
\begin{equation}\nonumber 
H_{\alpha\beta}:=
h (\partial_{\alpha},\partial_{\beta})
\stackrel{\eqref{hessian-metric-u}}{=}
\sum_{i,j=1}^n
\frac{\partial u^i}{\partial t^{\alpha}}
\frac{\partial u^j}{\partial t^{\beta}}
\frac{\partial^2 t^n}{\partial u^i\partial u^j}~,
\quad H=(H_{\alpha\beta})~.
\end{equation}
Comparing this with \eqref{trivial-connection},
we have
\begin{equation}\nonumber 
H_{\alpha\beta}=-\Omega_{\alpha\beta}^n~.
\end{equation}
Therefore from \eqref{omega-mr},
we obtain a key relation
\begin{equation}\label{H-B}
H_{\alpha\beta}=-\Omega_{\alpha\beta}^n=\frac{d_n-1}{d_1}B_{\alpha\beta}^n~.
\end{equation}

\subsection{Representation matrix of $\mathrm{Hess}(t^n)$}
Define vector fields $X_{\beta}$ ($1\leq \beta \leq n$) on $M$ by
\begin{equation}\label{Xbeta}
X_{\beta}=E\ast \partial_{\beta}=\sum_{\gamma=1}^n
{U^{\gamma}}_{\beta}\partial_{\gamma}
\end{equation}
They are $G$-invariant vector fields on $V$
and form a basis of the
$\mathbb{C}[V]^G$-module $\mathrm{Der}_{\mathbb{C}[V]}^G$ 
of $G$-invariant vector fields on $V$ \cite[\S 7.3]{KMS2018}. 
Take $dt^1,\ldots, dt^n$ as a basis of 
 the
$\mathbb{C}[V]^G$-module $\Omega_{\mathbb{C}[V]}^G$
of $G$-invariant $1$-forms on $V$.
Let $A=(A_{\alpha\beta})$ be the representation matrix of 
$\mathrm{Hess}(t^n)$  with respect to them:
\begin{equation}\label{def-A}
\mathrm{Hess}(t^n) (X_{\beta})=\sum_{\alpha=1}^n A_{\alpha \beta}\, dt^{\alpha}
~.
\end{equation}
Notice that
$$
A_{\alpha\beta}\in \mathbb{C}[t]~,\qquad 
\det A\neq 0~,
$$
since $\mathrm{Hess}(t^n)$ is the isomorphism 
(see (CS2) in \S \ref{section:CS-groups}).

\begin{lemma}\label{lemma:A-0}
$A=HU$  and 
$$
A_{\alpha\beta}=\frac{d_n-1}{d_1}C_{\alpha\beta}^n \quad 
\in \mathbb{C}[t']
$$ 
\end{lemma} 

\begin{proof} First we show $A=HU$.
By definition of $A_{\alpha\beta}$, 
\begin{equation}\begin{split}\nonumber
A_{\alpha\beta}=&\langle
\mathrm{Hess}(t^n) (X_{\beta}),\partial_{\alpha}
\rangle
\stackrel{\eqref{Xbeta}}{=}
\sum_{\gamma=1}^n  {U^{\gamma}}_{\beta}
\langle
\mathrm{Hess}(t^n)(\partial_{\gamma}),\partial_{\alpha}
\rangle 
\\
&=\sum_{\gamma=1}^n {U^{\gamma}}_{\beta} \, h(\partial_{\gamma},\partial_{\alpha})
=
\sum_{\gamma=1}^n {U^{\gamma}}_{\beta}H_{\gamma\alpha}
=
\sum_{\gamma=1}^n {U^{\gamma}}_{\beta}H_{\alpha\gamma}
\\&=(HU)_{\alpha\beta}~.
\end{split}
\end{equation}

To show the second statement, recall \eqref{H-B}. 
\begin{equation}\begin{split}\nonumber
A_{\alpha\beta}&=\sum_{\gamma=1}^n H_{\alpha\gamma}{U^{\gamma}}_{\beta}
\stackrel{\eqref{H-B}}{=}
\frac{d_n-1}{d_1}{(B_{\alpha}U)^n}_{\beta}
\stackrel{\eqref{B-C}}{=}
\frac{d_n-1}{d_1}\,C_{\alpha\beta}^{n}~.
\end{split}
\end{equation}
Then
$A_{\alpha\beta}\in \mathbb{C}[t']$ follows from 
\eqref{indep-t1}.
\end{proof}

\begin{lemma}\label{lemma:A-1}
\begin{enumerate}
\item 
$
\mathrm{deg}\,A_{\alpha\beta}=d_1+d_n-d_{\alpha}-d_{\beta}
$
and 
$$
\begin{cases}
A_{\alpha\beta}=0 &(\alpha+\beta<n+1)\\
A_{\alpha\beta}\in \mathbb{C}\setminus\{0\} &(\alpha+\beta=n+1)
\\
A_{\alpha\beta}\in \mathbb{C}[t'] &(\alpha+\beta>n+1)
\end{cases}
$$
\item $A^{-1}\in M_n (\mathbb{C}[t'])$. 
\item 
$A^{-1}WA\in  M_n (\mathbb{C}[t'])$ is upper triangular 
and its diagonal entries are
$$
{(A^{-1}WA)^{\mu}}_{\mu}=\frac{d_{n+1-\mu}}{d_1}=\frac{d_1+d_n-d_{\mu}}{d_1}~.
$$
\item $A^{-1}\partial_{\alpha} A\in  M_n (\mathbb{C}[t'])$ is strictly upper triangular. 
\end{enumerate}
\end{lemma}

\begin{proof}
(1) 
By Lemma \ref{lemma:A-0} and \eqref{degree-C}, 
the degree of $A_{\alpha\beta}$ is 
$$
\mathrm{deg}\, A_{\alpha\beta}
=
\mathrm{deg}\, C_{\alpha\beta}^n
=d_1+d_n- d_{\alpha}-d_{\beta}~.
$$
Recall that
 $d_{\alpha}+d_{\beta}>d_1+d_n$  holds
if $\alpha+\beta<n+1$
(see \eqref{cs-degrees}).
Therefore if $\alpha+\beta<n+1$, $\mathrm{deg}\, A_{\alpha\beta}<0$, hence $A_{\alpha\beta}=0$.
Recall also that $d_{\alpha}+d_{\beta}=d_1+d_n$
if $\alpha+\beta=n+1$ (see (CS3) in \S \ref{section:CS-groups}). 
Therefore $A_{\alpha\beta}$ is a constant if 
$\alpha+\beta=n+1$. 
Then we have
$$
\det A=(-1)^{\frac{n(n-1)}{2}}\prod_{\alpha=1}^n  A_{\alpha,n+1-\alpha}~.
$$
Since $\det A\neq 0$, $A_{\alpha,n+1-\alpha}\neq 0$.
\\
(2) Let us put 
$$
T=\begin{pmatrix}&&1\\&\iddots&\\1&&\end{pmatrix}~.
$$
Then $TA$ is the matrix obtained by exchanging 
the $i$-th row and the $(n+1-i)$-th row ($1\leq i\leq n$) of $A$.
So $TA\in M_n(\mathbb{C}[t'])$ 
is 
upper triangular and its diagonal entries 
are nonzero constants $A_{n1},\ldots, A_{1n}$.
Therefore $TA$ is invertible,
$(TA)^{-1}=A^{-1}T \in M_n(\mathbb{C}[t'])$
is upper triangular and its diagonal entries 
are nonzero constants. 
$A^{-1}$ is obtained from $A^{-1}T$ by exchanging
the $j$-th column and the $(n+1-j)$-th column ($1\leq j\leq n$). 
Therefore $A^{-1}\in M_n( \mathbb{C}[t'])$.
\\
(3) and (4) immediately follow from 
$A^{-1}WA=(TA)^{-1}(TWT)(TA)$
and $A^{-1}\partial_{\alpha}A=(TA)^{-1}\partial_{\alpha}(TA)$.
\end{proof}

\begin{lemma}\label{lemma:A-2}
$A={}^t A$ and 
$$
AC_{\alpha}={}^tC_{\alpha}A~,\quad 
 AB_{\alpha}={}^tB_{\alpha}A
\quad (1\leq \alpha\leq n)~,
\quad
AU={}^tU A~.
$$
\end{lemma}

\begin{proof}
By Lemma \ref{lemma:A-0},  we have
$$
A_{\mu\nu}=\frac{d_n-1}{d_1}C_{\mu\nu}^n
\stackrel{\eqref{c1}}{=}\frac{d_n-1}{d_1}C_{\nu\mu}^n
=A_{\nu\mu}~.
$$
We also have
$$
(AC_{\alpha})_{\mu\nu}=\frac{d_n-1}{d_1}{(C_{\mu}C_{\alpha})^n}_{\nu}
\stackrel{\eqref{c1}}{=}
\frac{d_n-1}{d_1}{(C_{\alpha}C_{\mu})^n}_{\nu}
=\frac{d_n-1}{d_1}\sum_{\lambda=1}^n C_{\alpha\lambda}^nC_{\mu\nu}^{\lambda}~,
$$
and 
\begin{equation}\begin{split}\nonumber
(AB_{\alpha})_{\mu\nu}&=
\frac{d_n-1}{d_1}{(C_{\mu}B_{\alpha})^n}_{\nu}
\stackrel{\eqref{id-B}}{=}
\frac{d_n-1}{d_1}{(B_{\alpha}C_{\mu})^n}_{\nu}
=\frac{d_n-1}{d_1}\sum_{\lambda=1}^n 
B_{\alpha\lambda}^n C_{\mu\nu}^{\lambda}~.
\end{split}
\end{equation}
In all of the above equations,
the RHS's are symmetric with respect to the exchange of $\mu$
and $\nu$. So $A$, 
$AC_{\alpha}$  and $AB_{\alpha}$ are  symmetric matrices.  Therefore  $A={}^tA$, 
$AC_{\alpha}={}^t(AC_{\alpha})={}^tC_{\alpha}A$
and $AB_{\alpha}={}^t(AB_{\alpha})={}^tB_{\alpha}A$.
The remaining equation $AU={}^t UA $ easily follows
from $AC_{\alpha}={}^tC_{\alpha}A$ and \eqref{U-t}. 
\end{proof}

\subsection{Levi--Civita connections for Shephard groups}

In the $t$-coordinates,
the Levi--Civita connection $\boldsymbol{\nabla}^{cs}$
of the metric $h$
is expressed as
\begin{equation} \nonumber 
\begin{split}
\boldsymbol{\nabla}^{cs}_{\partial_{\alpha}} 
(\partial_{\beta})
&=\sum_{\gamma=1}^n S_{\alpha\beta}^{\gamma} 
\partial_{\gamma}
~,
\quad
S_{\alpha\beta}^{\gamma}=
\frac{1}{2}\sum_{\delta=1}^n H^{\gamma\delta}
\left(
\partial_{\alpha} H_{\delta\beta}
+\partial_{\beta} H_{\delta\alpha}
-\partial_{\delta} H_{\alpha\beta}
\right)~.
\end{split}
\end{equation}
Here $H=(H_{\alpha\beta})$ is 
the representation matrix of the metric $h$
defined in \S \ref{rep-h} and
$H^{\alpha\beta}$ ($1\leq \alpha,\beta\leq n$) denotes
the $(\alpha,\beta)$ entry of the inverse matrix $H^{-1}$.
We put
$$
{(S_{\alpha})^{\gamma}}_{\beta}=S_{\alpha\beta}^{\gamma}~.
$$
\begin{lemma}\label{lemma-S}
$$
2S_{\alpha}=A^{-1}\partial_{\alpha}A+
(-I-W+A^{-1}WA)B_{\alpha}~.
$$
\end{lemma}

\begin{proof}
With \eqref{H-B} and \eqref{dB}, we have
\begin{equation}\begin{split}\nonumber
\partial_{\alpha} H_{\delta\beta}&=
\frac{d_n-1}{d_1}\partial_{\alpha}B_{\delta\beta}^n
\\&=\frac{d_n-1}{d_1}{\big((\partial_{\alpha}C_{\delta}) U^{-1}-B_{\delta}WB_{\alpha}
+U^{-1}C_{\delta}C_{\alpha}WU^{-1}-B_{\delta}B_{\alpha}\big)^{n}}_{\beta}~.
\end{split}
\end{equation}
Therefore
\begin{equation}\begin{split}\nonumber
\partial_{\alpha} H_{\delta\beta}- \partial_{\delta} H_{\alpha\beta}
&
\stackrel{\eqref{c1} \eqref{id-B}}{=}\frac{d_n-1}{d_1}{(B_{\alpha}WB_{\delta}-B_{\delta}WB_{\alpha})^{n}}_{\beta}
\\
&\stackrel{\eqref{H-B}}{=}(HWB_{\delta})_{\alpha\beta}-(HWB_{\alpha})_{\delta\beta}
\\&\stackrel{\eqref{id-B}}{=}(HWB_{\beta})_{\alpha\delta}-(HWB_{\alpha})_{\delta\beta}~.
\end{split}
\end{equation}
On the other hand, by \eqref{H-B} and 
Lemma \ref{lemma:A-0},
\begin{equation}\begin{split}\nonumber
\partial_{\beta} H_{\delta\alpha}&=
\partial_{\beta} H_{\alpha\delta}
=\frac{d_n-1}{d_1}\partial_{\beta}B_{\alpha\delta}^n
\\&=\frac{d_n-1}{d_1}{(\partial_{\beta}C_{\alpha}\,U^{-1}-B_{\alpha}W B_{\beta}+
U^{-1}C_{\alpha}C_{\beta}WU^{-1}-B_{\alpha}B_{\beta})^n}_{\delta}
\\&=
\frac{d_n-1}{d_1}{(\partial_{\alpha}C_{\beta}\,U^{-1}-B_{\alpha}W B_{\beta}+
C_{\beta}B_{\alpha}WU^{-1}-C_{\beta}B_{\alpha}U^{-1})^n}_{\delta}
\\
&=
((\partial_{\alpha}A) U^{-1})_{\beta\delta}
-(HWB_{\beta})_{\alpha\delta}+
(AB_{\alpha}WU^{-1})_{\beta\delta}
-(AB_{\alpha}U^{-1})_{\beta\delta}
\end{split}
\end{equation}
Adding these two equations, we obtain
\begin{equation}\nonumber
\partial_{\alpha}H_{\delta\beta}+\partial_{\beta}H_{\delta\alpha}
-\partial_{\delta}H_{\alpha\beta}
=\big((\partial_{\alpha}A+AB_{\alpha}(W-I))U^{-1} \big)_{\beta\delta}
-(HWB_{\alpha})_{\delta\beta}~.
\end{equation}
Therefore 
\begin{equation}\nonumber
\begin{split}
2S_{\alpha\beta}^{\gamma}&=
\sum_{\delta=1}^n
H^{\gamma\delta}
\big((\partial_{\alpha}A+AB_{\alpha}(W-I))U^{-1} \big)_{\beta\delta}
-\sum_{\delta=1}^n
H^{\gamma\delta}(HWB_{\alpha})_{\delta\beta}
\\&=
{\big((\partial_{\alpha}A+AB_{\alpha}(W-I))U^{-1} H^{-1}
\big)_{\beta}}^{\gamma}
-{(WB_{\alpha})^{\gamma}}_{\beta}
\\
&={\big((\partial_{\alpha}A+AB_{\alpha}(W-I))A^{-1}
\big)_{\beta}}^{\gamma}
-{(WB_{\alpha})^{\gamma}}_{\beta}~.
\end{split}
\end{equation}
Using Lemma \ref{lemma:A-2},  we see that the matrix in the first term is 
the transpose of 
$$
A^{-1}\partial_{\alpha}A+A^{-1}(W-I)AB_{\alpha}~.
$$
Therefore
$$
2S_{\alpha\beta}^{\gamma}={(A^{-1}\partial_{\alpha}A+(
-I-W+A^{-1}WA)B_{\alpha})^{\gamma}}_{\beta}~.
$$
\end{proof}

\section{Proofs of 
Theorem \ref{theorem1},
Theorem \ref{cs-polynomial} and 
Theorem \ref{theorem:multiplication}}
\label{section:proof-multiplication}
\subsection{Proofs of 
Theorem \ref{theorem1}-(1)(2) and Theorem   \ref{theorem:multiplication}}
\label{proof1}
First we show that $$\det S_1\neq 0$$ holds on $M^{\circ}$.
Substituting  $\partial_1 A=O$ and $B_1=U^{-1}$ into Lemma \ref{lemma-S}, we have
\begin{equation}\label{S1}
 2S_1= (-I-W+A^{-1}WA)U^{-1}~.
\end{equation}
By Lemma \ref{lemma:A-1} (3), 
$-I-W+A^{-1}WA$ is upper triangular and 
\begin{equation}\nonumber
{(-I-W+A^{-1}WA)^{\mu}}_{\mu}=-1-\frac{d_{\mu}}{d_1}+
\frac{d_1+d_n-d_{\mu}}{d_1}
=\frac{-2d_{\mu}+d_n}{d_1}\stackrel{\eqref{degrees}}{<}0~.
\end{equation}
This implies $\det (-I-W+A^{-1}WA)\neq 0$.
Since $\det U\neq 0$ on $M^{\circ}$,
$\det S_1\neq 0$ on $M^{\circ}$.
The representation matrix of the map $X\to \boldsymbol{\nabla}^{cs}_X e$ is given by $S_1$. 
Therefore $\det S_1 \neq 0$ implies that 
this map is invertible. This proves  
Theorem \ref{theorem1}-(1).
Thus \eqref{ass4-S} determines the multiplication $\diamond$. 

Next we prove Theorem \ref{theorem:multiplication}.
Let $B^{cs}_{\alpha}$ $(1\leq \alpha\leq n)$ denote the 
representation matrix of $\partial_{\alpha}\diamond$.
To show that the multiplication $\diamond$ 
agrees with $\star$, it is enough to show that 
$B_{\alpha}^{cs}=B_{\alpha}$.
Notice that \eqref{ass4-S} is written as follows.
\begin{equation}\nonumber 
O=\partial_{\alpha}S_1+S_{\alpha}S_1-S_1S_{\alpha}+S_1B_{\alpha}^{cs}=
\partial_1 S_{\alpha}+S_1 B_{\alpha}^{cs}~.
\end{equation}
In the last line, we used  the flatness of $\boldsymbol{\nabla}^{cs}$.
Therefore 
\begin{equation}\label{Bcs}
B_{\alpha}^{cs}=-S_1^{-1}\partial_1 S_{\alpha}~.
\end{equation}
Since $A$ is independent of $t^1$ (Lemma \ref{lemma:A-1}), we have
\begin{equation}\nonumber 
2\partial_1 S_{\alpha}
=(-I-W+A^{-1}WA) \partial_1 B_{\alpha}
\stackrel{\eqref{dB2}}{=}-(-I-W+A^{-1}WA) U^{-1}B_{\alpha}~.
\end{equation}
Substituting this equation and \eqref{S1} into \eqref{Bcs}, we obtain
$
B_{\alpha}^{cs}
=B_{\alpha}.
$

Theorem \ref{theorem1}-(2) immediately follows from Theorem \ref{theorem:multiplication}.

\subsection{Proof of Theorem \ref{theorem1} (3)}
We show that $(\boldsymbol{\nabla}^{cs},\diamond=\star,E)$ satisfies ({\bf ASS1})--({\bf ASS4}). 
We already showed that $\boldsymbol{\nabla}^{cs}$ and $E$
satisfy ({\bf ASS3}) in Lemma \ref{cs-flat}. 
It is clear that $\diamond$ satisfies ({\bf ASS2})
since $\star=\diamond$ satisfies ({\bf ASS2}).
It is also clear that ({\bf ASS4}) holds since
the multiplication $\diamond$ is made from 
$\boldsymbol{\nabla}^{cs}$ and $e$ by the condition
({\bf ASS4}). 
So
we only have to check the condition ({\bf ASS1}). 
 
In the matrix representation, 
({\bf ASS1}) is equivalent to 
$$
\partial_{\alpha}B_{\beta}+[S_{\alpha},B_{\beta}]
=\partial_{\beta}B_{\alpha}+[S_{\beta},B_{\alpha}]. 
$$
Using Lemma \ref{lemma-S},
this is equivalent to 
\begin{equation}\label{to-show-ass1}
\begin{split}
2\left( 
\partial_{\alpha}B_{\beta}-\partial_{\beta}B_{\alpha}
\right)
&=-[A^{-1}\partial_{\alpha}A,B_{\beta} ]+
[A^{-1}\partial_{\beta}A,B_{\alpha} ]
\\&+B_{\beta}(-W+A^{-1}WA)B_{\alpha}
-B_{\alpha}(-W+A^{-1}WA)B_{\beta}~.
\end{split}
\end{equation}
To show \eqref{to-show-ass1},
let us compute $\partial_{\alpha}B_{\beta}-\partial_{\beta}B_{\alpha}$ in two ways.
By \eqref{dB},  we have
 \begin{equation}\nonumber
\partial_{\alpha}B_{\beta}-\partial_{\beta}B_{\alpha}=
-B_{\beta}W B_{\alpha}+B_{\alpha}WB_{\beta}~.
\end{equation}
On the other hand, using $B_{\alpha}=A^{-1}({}^tB_{\alpha})A$
(Lemma \ref{lemma:A-2}), we have 
\begin{equation}\nonumber
\begin{split}
&\partial_{\alpha}B_{\beta}-\partial_{\beta}B_{\alpha}=
\partial_{\alpha}(A^{-1}\,{}^t B_{\beta} A) -\partial_{\beta}
(A^{-1}\,{}^t B_{\alpha} A)
\\&=-A^{-1}\partial_{\alpha}A \underbrace{A^{-1}{}^tB_{\beta}A}_{=B_{\beta}}
+A^{-1}\partial_{\alpha}{}^tB_{\beta} A+
\underbrace{A^{-1}B_{\beta}}_{=B_{\beta}A^{-1}}\partial_{\alpha}A
-(\alpha\leftrightarrow \beta)
\\
&=-[A^{-1}\partial_{\alpha}A,B_{\beta}]+
[A^{-1}\partial_{\beta}A,B_{\alpha}]
+A^{-1}{}^t(\partial_{\alpha} B_{\beta}-\partial_{\beta}B_{\alpha})
A^{-1}
\\&=
-[A^{-1}\partial_{\alpha}A,B_{\beta}]+
[A^{-1}\partial_{\beta}A,B_{\alpha}]
-B_{\alpha} A^{-1}WA B_{\beta}+
 B_{\beta} A^{-1}WA B_{\alpha}~.
\end{split}
\end{equation}
Adding these two equations, we obtain
\eqref{to-show-ass1}.
\subsection{Proof of Theorem \ref{theorem1} (4)}

To show that $(h,\star, e)$ is an almost Frobenius structure,
we have to check \eqref{af1}, \eqref{af2}, \eqref{af3}
with $\boldsymbol{\nabla},g$ replaced by $\boldsymbol{\nabla}^{cs}, h$.
Eq. \eqref{af1} trivially holds since $\boldsymbol{\nabla}^{cs}$ is the 
Levi--Civita connection of $h$.
In the matrix form, 
\eqref{af2} and \eqref{af3} are equivalent to the followings.
$$
\sum_{\lambda=1}^n B_{\alpha\beta}^{\lambda}H_{\lambda\gamma}
=\sum_{\lambda=1}^n B_{\beta\gamma}^{\lambda}H_{\alpha\lambda}~, \quad 
\partial_1 H+HB_1=O~.
$$
But these immediately follows from \eqref{H-B},
\eqref{id-B} and \eqref{dB2}.

\subsection{Proof of Theorem \ref{cs-polynomial}}
\label{cs-polynomial-proof}
The proof is almost the same as the proof of \cite[Theorem 7.5-(3)]{KMS2018}. 

Consider
the Saito structure $(\nabla^{cs},\circledast, e)$ dual to $(\boldsymbol{\nabla}^{cs}, \diamond=\star, E)$.
Comparing the first equations of 
\eqref{duality-dual} and \eqref{CS-dual}, 
we see that the multiplication $\circledast$ agrees with
the multiplication $\ast$ of the natural Saito structure for $G$.
To show that $(\nabla^{cs},\circledast=\ast, e)$ 
is a polynomial Saito structure, we will find 
a set of basic invariants $s=(s^1,\ldots, s^n)$ satisfying
the following (i) and (ii): 
\begin{enumerate}
\item[(i)]
 $s=(s^1,s^2,\ldots, s^n)$ 
is a system of 
$\nabla^{cs}$-flat coordinates.
\item[(ii)]
The structure constants of the multiplication $\circledast=\ast$
with respect to the basis
$\frac{\partial}{\partial s^1},\ldots,\frac{\partial}{\partial s^n}$ 
are polynomials in $s$.
\end{enumerate}

Since $\circledast=\ast$,
the representation matrix of 
$\partial_{\alpha}\circledast $ is $C_{\alpha}$.
Let
$$
\nabla^{cs}_{\partial_{\alpha}}(\partial_{\beta})=\sum_{\gamma=1}^n \Upsilon_{\alpha\beta}^{\gamma}~,
\quad {(\Upsilon_{\alpha})^{\gamma}}_{\beta}=\Upsilon_{\alpha\beta}^{\gamma}~.
$$
Then by the second relation of \eqref{CS-dual}
and Lemma \ref{lemma-S},
\begin{equation}\label{upsilon}
\Upsilon_{\alpha}=
S_{\alpha}-S_1 C_{\alpha}=\frac{1}{2} A^{-1}\partial_{\alpha}A
\quad (1\leq \alpha\leq n)~.
\end{equation}
Notice that
Lemma \ref{lemma:A-1} implies that
$\Upsilon_{\alpha}$ 
is strictly upper triangular and that
$\mathrm{deg}\,\Upsilon_{\alpha\beta}^{\gamma}=d_{\gamma}
-d_{\alpha}-d_{\beta}$.  
Moreover, the flatness\footnote{
The connection $\nabla^{cs}$ is flat
since it is constructed from $\boldsymbol{\nabla}^{cs}$
by \eqref{CS-dual}.
See \cite[Proposition 3.7]{KMS2018}.} 
of $\nabla^{cs}$ implies that 
\begin{equation}\label{upsilon-flat}
\partial_{\alpha}\Upsilon_{\beta}-\partial_{\beta}\Upsilon_{\alpha}
+[\Upsilon_{\alpha},\Upsilon_{\beta}]=O~.
\end{equation}
\begin{lemma}\label{lemma-X}
There exists a unique upper unitriangular matrix 
$X\in M_n(\mathbb{C}[t'])$ with homogeneous entries 
satisfying
\begin{equation}\label{def-X}
\partial_{\alpha}X+\Upsilon_{\alpha} X=O
\quad (1\leq \alpha\leq n)~.
\end{equation}
Moreover $\mathrm{deg}\,{X^{\gamma}}_{\beta}=d_{\gamma}-d_{\beta}$.
\end{lemma}

\begin{proof}
Let us set  
$${X^{\gamma}}_{\gamma}=1\quad (1\leq \gamma\leq n)~,
\quad
{X^{\gamma}}_{\beta}=0\quad (1\leq \beta<\gamma\leq n)~.
$$
We will solve the equation \eqref{def-X}.
Component-wise, it is written as
\begin{equation}\label{def-X-2}
\partial_{\alpha}{X^{\gamma}}_{\beta}=-
\sum_{\delta=1}^n
\Upsilon_{\alpha\delta}^{\gamma} {X^{\delta}}_{\beta}\quad
(1\leq \alpha\leq n)~.
\end{equation}
Notice that 
the sum in the RHS is taken for $\gamma<\delta\leq \beta$
since $\Upsilon_{\alpha\delta}^{\gamma}=0$ if 
$\gamma\geq \delta$ and 
${X^{\delta}}_{\beta}=0$ if $\delta>\beta$.

Now let us fix $1\leq \beta\leq n$. 
If $\gamma\geq \beta$,  the RHS of \eqref{def-X-2} is zero
because $\delta$ satisfying $\gamma<\delta\leq \beta$
does not exist.
The LHS is also zero since ${X^{\gamma}}_{\beta}=0$ or $1$. So \eqref{def-X-2} holds.  

Consider the case $\gamma=\beta-1$. 
The system of partial differential equations 
\eqref{def-X-2} for ${X^{\beta-1}}_{\beta}$ becomes
$$
\partial_{\alpha}{X^{\beta-1}}_{\beta}=-
\Upsilon_{\alpha\beta}^{\beta-1} \quad (1\leq \alpha\leq n)~.
$$
The homogeneous polynomial solution ${X^{\beta-1}}_{\beta}\in \mathbb{C}[t]$ uniquely exists
due to  \eqref{upsilon-flat}.
Its degree is 
$$
\mathrm{deg}\,{X^{\beta-1}}_{\beta}=\mathrm{deg}\,
\Upsilon_{\alpha\beta}^{\beta-1} +d_{\alpha}=
(d_{\beta-1}-d_{\alpha}-d_{\beta})+d_{\alpha}
=d_{\beta-1}-d_{\beta}~.
$$
For $\gamma=\beta-2$,
\eqref{def-X-2} becomes
$$
\partial_{\alpha}{X^{\beta-2}}_{\beta}=-
\Upsilon_{\alpha\beta}^{\beta-2}
-\Upsilon_{\alpha,\beta-1}^{\beta-2} 
{X^{\beta-1}}_{\beta}
\quad (1\leq \alpha\leq n)~.
$$
The homogeneous polynomial solution ${X^{\beta-2}}_{\beta}\in \mathbb{C}[t]$ uniquely exists
due to  \eqref{upsilon-flat} and 
its degree is  $d_{\beta-2}-d_{\beta}$.
For $\gamma=\beta-3,\ldots, 2,1$,
the similar argument shows the existence of 
homogeneous polynomial solution ${X^{\gamma}}_{\beta}\in\mathbb{C}[t]$
of degree $d_{\gamma}-d_{\beta}$.

Since $\mathrm{deg}\,{X^{\gamma}}_{\beta}<d_{\gamma}\leq d_1$,
${X^{\gamma}}_{\beta}$ is independent of $t^1$, i.e.
${X^{\gamma}}_{\beta}\in \mathbb{C}[t']$.
\end{proof}

Now let $X\in M_n(\mathbb{C}[t'])$ be the matrix
in Lemma \ref{lemma-X}.
Since $X\in M_n(\mathbb{C}[t'])$ is upper unitriangular,
$X$ is invertible and 
$X^{-1}\in M_n(\mathbb{C}[t'])$
is also upper unitriangular.
Moreover $\mathrm{deg}{(X^{-1})^{\gamma}}_{\beta}=d_{\gamma}-d_{\beta}$. 
Then we can find homogeneous polynomials 
$s^1,\ldots, s^n \in \mathbb{C}[t]$ 
satisfying
$$
ds^{\alpha}=\sum_{\beta=1}^n {(X^{-1})^{\alpha}}_{\beta} dt^{\beta},\quad 
\mathrm{deg}\,s^{\alpha}=d_{\alpha}\quad
(1\leq \alpha\leq n)~.
$$
By degree consideration, 
$s^1,\ldots, s^n$ are of the following forms:
$$s^n=t^n~,\quad s^{n-1}=t^{n-1}+F_{n-1}(t^n)~,
\quad\ldots,\quad
s^1=t^1+F_1(t^2,\ldots, t^n)~.
$$
We can solve these equations for $t$ and
express $t^1,\ldots, t^n$ as 
polynomials in $s^1,\ldots, s^n$. 
Therefore $s=(s^1,\ldots, s^n)$ is a set of basic invariants.
We obtain $\mathbb{C}[s]=\mathbb{C}[t]$
and $\mathbb{C}[s']=\mathbb{C}[t']$.

Next let us show (i). 
Taking the dual of $ds^1,\ldots, ds^n$,  we obtain
\begin{equation}\label{vector-s}
\frac{\partial}{\partial s^{\beta}}=\sum_{\gamma=1}^n{X^{\gamma}}_{\beta}{\partial \gamma}\quad 
(1\leq \beta\leq n)~.
\end{equation}
Applying $\nabla^{cs}_{\partial_{\alpha}}$, we have
$$
\nabla^{cs}_{\partial_{\alpha}} \frac{\partial}{\partial s^{\beta}}
=\sum_{\delta=1}^n {(
\partial_{\alpha }X +\Upsilon_{\alpha}X)^{\delta}}_{\beta}
\partial_{\delta}
\stackrel{\eqref{def-X}}{=}0~.
$$
Thus $s=(s^1,\ldots,s^n)$ is a system of $\nabla^{cs}$-flat
coordinates.

Finally we show (ii).
Denote by $\hat{C}_{\alpha\beta}^{\gamma}$
the structure constants of $\circledast=\ast$ 
with respect to the new basis $\frac{\partial}{\partial s^1},\ldots,
\frac{\partial}{\partial s^n}$:
$$
 \frac{\partial}{\partial s^{\alpha}}\ast
 \frac{\partial}{\partial s^{\beta}}
 =\sum_{\gamma=1}^n \hat{C}_{\alpha\beta}^{\gamma}
 \frac{\partial}{\partial s^{\gamma}}~.
$$
 Then by \eqref{vector-s},
 $$
  \hat{C}_{\alpha\beta}^{\gamma}=
 \sum_{\mu,\nu,\lambda}
 {X^{\mu}}_{\alpha}{X^{\nu}}_{\beta} C_{\mu\nu}^{\lambda}
 {(X^{-1})^{\gamma}}_{\lambda}~.
 $$
 Since $X,X^{-1},C_{\mu}\in M_n(\mathbb{C}[t'])$, 
 $\hat{C}_{\alpha\beta}^{\gamma}\in \mathbb{C}[t']=\mathbb{C}[s']$. 
Theorem \ref{cs-polynomial} is proved.

\section{Proof of Theorem \ref{theorem3}}
\label{proof-theorem3}
The notations are the same as 
\S \ref{section:proof-multiplication}.

\begin{lemma}\label{lemma:A-3-2}
$\nabla^{cs}=\nabla$ holds if and only if 
$A$ is an anti-diagonal  constant matrix.
\end{lemma}

\begin{proof}
The condition $\nabla^{cs}=\nabla$  is equivalent to  
$\Upsilon_{\alpha}=O$ ($1\leq \alpha\leq n$). 
Therefore by \eqref{upsilon} 
$\nabla^{cs}=\nabla$
holds if and only if $A$ is a constant matrix.
By Lemma \ref{lemma:A-1},
$A$ is a constant matrix if and only if 
it is anti-diagonal.
\end{proof}

\begin{remark}
In the case $n=2$,  
$A$ is given by
$$
A=\frac{d_2-1}{d_1}\begin{pmatrix}0&1\\
1&C_{22}^2 
\end{pmatrix}~.
$$ 
Therefore by Lemma \ref{lemma:A-3-2}, 
$A$ is an anti-diagonal  if and only if 
$C_{22}^2=0$.
It is not difficult to compute $C_{22}^2$. 
See \cite[\S 5]{Arsie-Lorenzoni2016}, \cite[Tables C6, C7 ,C8]{KMS2018}.
Among exceptional Shephard groups of rank $2$, 
$C_{22}^2\neq 0$
holds only for
$G_4$, $G_5$, 
$G_8$, $G_{16}$, $G_{20}$.
This result agrees with Theorem \ref{theorem2}
proved by the calculation using the $u$-coordinates.
\end{remark}

Next we consider the metric $\eta$ defined by
\eqref{CS-dual-2}.
\begin{lemma}\label{lemma:A-3}
The matrix 
$A$ defined in \eqref{def-A} is the representation matrix of the metric $\eta$ with respect to 
$\partial_{\alpha}$ $(1\leq \alpha\leq n)$: $$\eta(\partial_{\alpha},\partial_{\beta})=A_{\alpha\beta}\quad (1\leq \alpha,\beta\leq n).$$
\end{lemma}
\begin{proof}  Substituting $\circledast=\ast$,
$X=\partial_{\alpha}$
and $Y=\partial_{\beta}$ into \eqref{CS-dual-2},
we obtain
$$
\eta(\partial_{\alpha},\partial_{\beta})=
h(\partial_{\alpha},E\ast\partial_{\beta})
=h\left(\partial_{\alpha},
\sum_{\gamma=1}^n {U^{\gamma}}_{\beta}\partial_{\gamma}\right)=
(HU)_{\alpha\beta}~.
$$
Since $A=HU$ (see Lemma \ref{lemma:A-0}),  
$\eta(\partial_{\alpha},\partial_{\beta})=A_{\alpha\beta}$.
\end{proof}

Now we prove Theorem \ref{theorem3}.
Assume that 
$\theta$ is a metric on $M$
compatible with the natural Saito structure $(\nabla,\ast,E)$.
Then $\theta$ must satisfy \eqref{f1}, \eqref{f2}, \eqref{f3}
(with $\eta$ replaced by $\theta$).
Let $\Theta$ be the representation matrix of $\theta$
with respect to $\partial_{\alpha}$ $(1\leq \alpha\leq n)$, i.e. 
$$
\Theta_{\alpha\beta}=\theta(\partial_{\alpha},\partial_{\beta})~.
$$
Then \eqref{f1} is equivalent to 
$$\partial_{\alpha} \Theta_{\beta\gamma}=0 \quad 
 (1\leq \alpha,\beta,\gamma\leq n)~.
$$
Therefore $\Theta$ must be a constant matrix.
 
Eq.\eqref{f3} is equivalent to 
\begin{equation}\nonumber
\frac{d_{\alpha}+d_{\beta}}{d_1}\Theta_{\alpha\beta}
=(2-D)\Theta_{\alpha\beta}
\quad (1\leq \alpha,\beta\leq n)~.
\end{equation} 
Since $\Theta\neq O$, it follows that 
$2-D=(d_{\alpha}+d_{\beta})/d_1$ must holds for 
some $(\alpha,\beta)$.
Let us show that 
\begin{equation}\nonumber
2-D=\frac{d_1+d_n}{d_1}~.
\end{equation}
First assume that 
$
d_1(2-D)>d_1+d_n
$
holds. Then by \eqref{cs-degrees}, 
$$
\frac{d_{\alpha}+d_{\beta}}{d_1}\leq 
\frac{d_1+d_n}{d_1}<2-D \quad (\alpha+\beta\geq n+1)~.
$$
So $\Theta_{\alpha\beta}=0$ must hold for $\alpha+\beta\geq n+1$, which implies $\det \Theta=0$. 
This contradicts the nondegeneracy of the metric $\theta$. 
Therefore $d_1(2-D)\leq d_1+d_n$. 
By a similar argument,  we can show $d_1(2-D)\geq d_1+d_n$. 
Thus we have
\begin{equation}\label{theta}
2-D=\frac{d_1+d_n}{d_1}~,\quad 
\Theta_{\alpha\beta}=0 \quad (\alpha+\beta\neq n+1)~,\quad 
\Theta_{\alpha,n+1-\alpha}\in \mathbb{C}\setminus \{0\}~.
\end{equation}

Finally, the condition \eqref{f2} implies
$\theta(\partial_{\alpha},\partial_{\beta})=
\theta(\partial_{\alpha}\ast \partial_{\beta},\partial_1)$. So
$$
\Theta_{\alpha\beta}=
\sum_{\gamma=1}^n
C_{\alpha\beta}^{\gamma} \Theta_{\gamma,1}
\stackrel{\eqref{theta}}{=}C_{\alpha\beta}^n \Theta_{n,1}
\stackrel{\mathrm{Lemma\, \ref{lemma:A-0}}}{=}
\frac{d_1 \Theta_{n,1}}{d_n-1} A_{\alpha\beta}~. 
$$
Since $\Theta$ is anti-diagonal, 
this equation implies that 
$A$ must be  an anti-diagonal matrix.
Hence by by Lemma \ref{lemma:A-3-2},
$\nabla=\nabla^{cs}$ must hold. 
Moreover  
$\theta$ must be a constant multiple of $\eta$
since $A$ is the representation matrix of $\eta$ 
(Lemma \ref{lemma:A-3}).

The proof of  the converse is immediate. 
If we assume that $\nabla=\nabla^{cs}$, 
then it is clear that the constant multiple of $\eta$
is compatible with $(\nabla^{cs},\circledast=\ast,E)$
since $\eta$ is a metric compatible with 
$(\nabla^{cs},\circledast=\ast,E)$.
(See the last paragraph in \S \ref{section:CS-groups}.)
This finishes the proof of Theorem \ref{theorem3}.
\appendix
\section{Vector field $e$ for $G(m,1,n)$}\label{appendix1}

\subsection{Preliminary}
For a tuple of variables $v=(v^1,\ldots, v^n)$,
let
$$
{\bf e}_{\alpha}(v)=\sum_{1\leq i_1<i_2<\ldots<i_{\alpha}\leq n}
v^{i_1}v^{i_2}\cdots v^{i_{\alpha}}~
$$
be the $\alpha$-th elementary symmetric polynomial in  $v$.
We use the notation $$
v_{(i,j,\ldots,k)}=v\setminus \{v^i,v^j,\ldots, v^k\}~.
$$
Notice that 
\begin{equation}\label{elementary}
{\bf e}_{\alpha}(v_{(i,j,\ldots,k)})=
v^l {\bf e}_{\alpha-1}(v_{(i,j,\ldots,k,l)})
+{\bf e}_{\alpha}(v_{(i,j,\ldots,k,l)}) 
\quad (l\neq i,j,\ldots,k)~.
\end{equation}

We set 
$$
\mathbb{E}(v)=
\mathbb{E}(v^1,\ldots,v^n)=\begin{pmatrix}
1&\ldots&1\\
{\bf e}_{1}(v_{(1)})&\ldots& {\bf }{\bf e}_1(v_{(n)})\\
\vdots&&\vdots\\
{\bf e}_{n-1}(v_{(1)})&\ldots&{\bf }{\bf e}_{n-1}(v_{(n)})
\end{pmatrix}~.
$$

\begin{lemma}\label{lem:detE}
The determinant of $\mathbb{E}(v)$ is given by
$$
|\mathbb{E}(v)|=\prod_{1\leq k<l\leq n}(v^k-v^l)~.$$
\end{lemma}
\begin{proof}

Subtracting the first column from 
the $k$-th column ($k>1$) and  using \eqref{elementary}, we have
\begin{equation}\nonumber
\begin{split}
|\mathbb{E}(v^1,\ldots,v^n) |
&=
\left|
\begin{array}{cccc}
1&0&\cdots&0\\
{\bf e}_1(v_{(1)})&v^1-v^2&\cdots&v^1-v^n\\
{\bf e}_2(v_{(2)})&(v^1-v^2){\bf e}_1(v_{(1,2)})&\cdots&(v^1-v^n){\bf e}_1(v_{(1,n)})\\
\vdots&\vdots&&\vdots\\
{\bf e}_n(v_{(n)})&(v^1-v^2){\bf e}_{n-1}(v_{(1,2)})&\cdots&(v^1-v^n){\bf e}_{n-1}(v_{(1,n)})\\
\end{array}
\right|\\
&=\prod_{2\leq k \leq n}(v^1-v^k)\cdot
\left| \mathbb{E}(v^2,\ldots,v^n) \right|. 
\end{split}
\end{equation}
The claim follows by induction.
\end{proof}

Next we calculate the $(\alpha,j)$ minor of $\mathbb{E}(v)$.
Let $\mathbb{E}(v)_{\alpha,j}$ be the matrix
obtained by deleting the $\alpha$-th row and the $j$-th column of 
$\mathbb{E}(v)$. 
\begin{lemma} \label{minor}
The determinant of $\mathbb{E}(v)_{\alpha,j}$ is given by
\begin{equation}\nonumber
|\mathbb{E}(v)_{\alpha,j}|=(v^j)^{n-\alpha}
\prod_{\begin{subarray}{c}1\leq k<l\leq n;\\
k,l\neq j
\end{subarray}} (v^k-v^l)~.
\end{equation}
\end{lemma}
\begin{proof}
We may assume $j=1$, since
\begin{equation}\nonumber
\left| {\mathbb E}(v^1,v^2,\ldots,v^n)_{\alpha,j} \right|
=(-1)^{j-2}
\left| {\mathbb E}(v^j,v^2,\ldots,v^{j-1},v^1,v^{j+1},v^n)_{\alpha,1} \right|.
\end{equation}
First we consider the case $\alpha=1$.
Subtract the first column of  $\mathbb{E}(v)_{1,1}$  from the $k$-th column ($k>1$).
Then subtract the second column from the $k$-th column ($k>2$).
Continuing this process, we have 
\begin{equation}\nonumber
\begin{split}
|\mathbb{E}(v)_{1,1} |
&=
\left|
\begin{array}{cccc}
{\bf e}_1(v_{(2)})&{\bf e}_1(v_{(3)})&\cdots&{\bf e}_1(v_{(n)})\\
{\bf e}_2(v_{(2)})&{\bf e}_2(v_{(3)})&\cdots&{\bf e}_2(v_{(n)})\\
\vdots&\vdots&\ddots&\vdots\\
{\bf e}_{n-1}(v_{(2)})&{\bf e}_{n-1}(v_{(3)})&\cdots&{\bf e}_{n-1}(v_{(n)})
\end{array}
\right|\\
&=\prod_{3\leq k \leq n}(v^2-v^k)
\left|
\begin{array}{cccc}
{\bf e}_1(v_{(2)})&1&\cdots&1\\
{\bf e}_2(v_{(2)})&{\bf e}_1(v_{(2,3)})&\cdots&{\bf e}_1(v_{(2,n)})\\
\vdots&\vdots&\ddots&\vdots\\
{\bf e}_{n-1}(v_{(2)})&{\bf e}_{n-2}(v_{(2,3)})&\cdots&{\bf e}_{n-2}(v_{(2,n)})
\end{array}
\right|\\
&=
\prod_{3\leq k \leq n}(v^2-v^k)
\prod_{4\leq k \leq n}(v^3-v^k)
\left|
\begin{array}{ccccc}
{\bf e}_1(v_{(2)})&1&0&\cdots&0\\
{\bf e}_2(v_{(2)})&{\bf e}_1(v_{(2,3)})&1&\cdots&1\\
\vdots&\vdots&\ddots&\vdots\\
{\bf e}_{n-1}(v_{(2)})&{\bf e}_{n-2}(v_{(2,3)})&{\bf e}_{n-3}(v_{(2,3,4)})&\cdots&{\bf e}_{n-3}(v_{(2,3,n)})
\end{array}
\right|\\
&\vdots\\
&=\prod_{2\leq i<j\leq n} (v^i-v^j)\cdot
\underbrace{
\left|
\begin{array}{cccccc}
{\bf e}_1(v_{(2)})&1&0&\cdots&0&0\\
{\bf e}_2(v_{(2)})&{\bf e}_1(v_{(2,3)})&1&\cdots&0&0\\
{\bf e}_3(v_{(2)})&{\bf e}_2(v_{(2,3)})&{\bf e}_1(v_{(2,3,4)})&\cdots&0&0\\
\vdots&\vdots&\vdots&\ddots&\vdots&\vdots\\
{\bf e}_{n-2}(v_{(2)})&{\bf e}_{n-3}(v_{(2,3)})&{\bf e}_{n-4}(v_{(2,3,4)})&\cdots&{\bf e}_1(v_{(2,\ldots,n-1)})&1\\
{\bf e}_{n-1}(v_{(2)})&{\bf e}_{n-2}(v_{(2,3)})&{\bf e}_{n-3}(v_{(2,3,4)})&\cdots&{\bf e}_2(v_{(2,\ldots,n-1)})&{\bf e}_1(v_{(2,\ldots,n)})
\end{array}
\right|
}_{(\spadesuit)}
~.
\end{split}
\end{equation}
Next we eliminate the entries  below the diagonal of $(\spadesuit)$ by column operations. 
We start from the bottom row. For $k<n-1$, 
eliminate the $(n-1,k)$ entry
${\bf e}_{n-k}(v_{(2,\ldots,k+1)})=v^1v^{k+2}\cdots v^n$
using the $(n-1, n-1)$ entry
${\bf e}_1(v_{(2,\ldots,n)})=v^1$. Then $(n-2, n-2)$ entry becomes 
${\bf e}_1(v_{(2,\ldots,n-1)})-v_n=v_1$ and all the $(n-2, k)$ entries for $k<n-2$
are divisible by $v_1$. Repeating the elimination process, we obtain
\begin{equation}\nonumber
(\spadesuit)=\left|
\begin{array}{cccc}
v^1&1&&O \\
&\ddots&\ddots&\\
&&v^1&1\\
O&&&v^1
\end{array}
\right|
=(v^1)^{n-1}~.
\end{equation} 
This proves the formula for $\alpha=1$.
The cases $\alpha>1$ may be reduced to the case $\alpha=1$ as follows.
By column operations similar to those in Lemma \ref{lem:detE}, we have
\begin{equation}\nonumber
\begin{split}
|\mathbb{E}(v)_{\alpha,1} |
&=
\prod_{2\leq i \leq \alpha}
\left(\prod_{i<j\leq n}
(v^i-v^j)\right)\cdot
\left|
\begin{array}{ccc}
{\bf e}_1(v_{(2,\cdots,\alpha,\alpha+1)}) &\cdots &{\bf e}_1(v_{(2,\cdots,\alpha,n)})\\
\vdots&\ddots&\vdots\\
{\bf e}_{n-\alpha}(v_{(2,\cdots,\alpha,\alpha+1)}) &\cdots &{\bf e}_{n-\alpha}(v_{(2,\cdots,\alpha,n)})
\end{array}
\right|\\
&=\prod_{2\leq i \leq \alpha}
\left(\prod_{i<j\leq n}
(v^i-v^j)\right)
\cdot
\left|\mathbb{E}(v^1,v^{\alpha+1},\ldots,v^{n})_{1,1}\right|~.
\end{split}
\end{equation}
Thus the formula for $\alpha>1$ follows from that for $\alpha=1$. 
 \end{proof}

From the above two lemmas,
we obtain the following 
\begin{lemma}\label{inverse}
 The $(i,\alpha)$ entry of $\mathbb{E}(v)^{-1}$ is given by
$$
(\mathbb{E}(v)^{-1})_{i\alpha}=
(-1)^{\alpha+1} (v^i)^{n-{\alpha}}
\prod_{\begin{subarray}{c}1\leq l\leq n;\\
l\neq i\end{subarray}
}
(v_i-v_l)^{-1}~.
$$
\end{lemma}
\begin{proof}
Substitute the formulas in Lemmas \ref{lem:detE} and \ref{minor}
into 
$$
(\mathbb{E}(v)^{-1})_{i\alpha}=
\dfrac{(-1)^{i+\alpha}|\mathbb{E}(v)_{\alpha,i}|}{|\mathbb{E}(v)|}~.
$$
\end{proof}

\subsection{The vector field $e$ for $G(m,1,n)$}
If we set 
$$
v^i=(u^i)^m \quad (1\leq i\leq n)~,
$$
then
a set of basic invariants for $G(m,1,n)$ 
is given by
$$
x^{\alpha}={\bf e}_{n+1-\alpha}(v)\quad 
(1\leq \alpha\leq n)~.
$$
\begin{proposition}\label{du/dx}
We have
$$
\frac{\partial u^i}{\partial x^{\alpha}}
=\frac{(-1)^{n+\alpha}
(u^i)^{m(\alpha-2)+1}
}{m}
\prod_{\begin{subarray}{c}
1\leq l\leq n;\\
l\neq i
\end{subarray}}
(v_i-v_l)^{-1}~.
$$
\end{proposition}
\begin{proof}

From \eqref{elementary},  it is immediate to see that
$$
\frac{\partial}{\partial v^i} {\bf }{\bf e}_{\alpha}(v)
={\bf e}_{\alpha-1}(v_{(i)})~.
$$
Therefore, using the chain rule, we have
\begin{equation}\nonumber
\begin{split}
\begin{pmatrix}
\frac{\partial x^n}{\partial u^1}&\ldots&\frac{\partial x^n}{\partial u^n}\\
\vdots&&\vdots\\
\frac{\partial x^1}{\partial u^1}&\ldots&\frac{\partial x^1}{\partial u^n}
\end{pmatrix}
&=
\begin{pmatrix}
\frac{\partial x^n}{\partial v^1}&\ldots&\frac{\partial x^n}{\partial v^n}\\
\vdots&&\vdots\\
\frac{\partial x^1}{\partial v^1}&\ldots&\frac{\partial v^1}{\partial u^n}
\end{pmatrix}
\begin{pmatrix}
\frac{\partial v^1}{\partial u^1}&\ldots&\frac{\partial v^1}{\partial u^n}\\
\vdots&&\vdots\\
\frac{\partial v^n}{\partial u^1}&\ldots&\frac{\partial v^n}{\partial u^n}
\end{pmatrix}
\\
&=m\, \mathbb{E}(v) \, \, \mathrm{diag}((u^1)^{m-1},\ldots,(u^n)^{m-1})
\end{split}
\end{equation}
Therefore by the inverse function theorem,
\begin{equation}\nonumber 
\begin{pmatrix}
\frac{\partial u^1}{\partial x^n}&\ldots&\frac{\partial u^1}{\partial x^1}\\
\vdots&&\vdots\\
\frac{\partial u^n}{\partial x^n}&\ldots&\frac{\partial u^n}{\partial x^1}
\end{pmatrix}
=
\frac{1}{m}\mathrm{diag}((u^1)^{1-m},\ldots,(u^n)^{1-m})\,\mathbb{E}(v)^{-1}~.
\end{equation}
Comparing the $(i,n+1-\alpha)$ entries of the both sides,
we obtain
$$
\frac{\partial u^i}{\partial x^{\alpha}}
=\frac{(u^i)^{1-m}}{m}( \mathbb{E}(v)^{-1})_{i,n+1-{\alpha}}
$$
Thus the statement follows from 
Lemma \ref{inverse}. 
\end{proof}
\begin{corollary}\label{e-gm1n}
The vector field $e=\frac{\partial}{\partial x^1}$ is given as follows.
$$
e=\sum_{k=1}^n e^k\frac{\partial}{\partial u^k}~,\qquad
e^k=\frac{(-1)^{n+1}
}{m (u^k)^{m-1}}
\prod_{\begin{subarray}{c}1\leq l\leq n;\\
l\neq k
\end{subarray}}
(v_k-v_l)^{-1}~.
$$
\end{corollary}


\end{document}